\documentclass[11pt]{article}
\usepackage{array,latexsym,amsmath,amssymb,mathrsfs,graphicx,hyperref,mathpazo,soul,color,amsthm}
\hypersetup{pdftex,colorlinks=true,linkcolor=blue,citecolor=blue,filecolor=blue,urlcolor=blue}

\usepackage{url}
\usepackage{hyperref}

\def\tr{{\raise0pt\hbox{$\scriptscriptstyle\top$}}}

\newtheorem{theorem}{Theorem}[section]

\newtheorem{proposition}[theorem]{Proposition}
\newtheorem{corollary}[theorem]{Corollary}

\newtheorem{problem}[theorem]{Problem}

\newtheorem{question}[theorem]{Open Question}
\newtheorem{algorithm}[theorem]{Algorithm}

\title{\vspace{-1.5 cm}{\bf Diophantine equations: a systematic approach}}

\author{Bogdan Grechuk\footnote{School of Computing and Mathematical Sciences, University of Leicester, LE1 7RH, UK; bg83@leicester.ac.uk}}

\begin{document}

\maketitle

\begin{abstract}
This paper initiates a novel research direction in the theory of Diophantine equations: define an appropriate version of the equation's size, order all polynomial Diophantine equations starting from the smallest ones, and then solve the equations in that order. By combining a new computer-aided procedure with human reasoning, we solved Hilbert's tenth problem for all polynomial Diophantine equations of size less than $31$, where the size is defined in \cite{Z2018}. In addition, we solved this problem for all equations of size equal to $31$, with a single exception. Further, we solved Hilbert's tenth problem for all two-variable Diophantine equations of size less than $32$, all symmetric equations of size less than $39$, all three-monomial equations of size less than $45$, and, in each category, identified the explicit smallest equations for which the problem remains open. As a result, we derived a list of equations that are very simple to write down but which are apparently difficult to solve. As we know from the example of Fermat's Last Theorem, such equations have the potential to stimulate the development of new methods in number theory.  
\end{abstract}



\section{Introduction.}

A polynomial Diophantine equation is an equation of the form
\begin{equation}\label{eq:diofgen}
P(x_1,\dots, x_n)=0,
\end{equation}
where $P$ is a polynomial with integer coefficients. The name ``Diophantine'' originates from the fact that such equations were studied by Diophantus of Alexandria, a mathematician of the 3rd century. Diophantine equations are a very active area of current research, and reviewing even a small portion of the existing literature goes beyond the scope of this paper. Instead, we refer the reader to the classical book of Dickson \cite{dickson2013history} that gives a systematic survey of essentially all research in this area up to about 1920, the 1969 book of Mordell \cite{mordell1969diophantine} that gives an excellent description of the main techniques for solving Diophantine equations, the great survey paper of Cassels \cite{cassels1966diophantine} , and more recent excellent books of Cohen \cite{cohen2008number} and Andreescu \cite{andreescu2010introduction}.

One of the basic problems in the theory of Diophantine equations is the following one.

\begin{problem}\label{prob:main}
Does equation \eqref{eq:diofgen} have an integer solution?
\end{problem}

In 1900, Hilbert \cite{hilbert1902mathematical} presented a list of $23$ mathematical problems for the next centuries. Hilbert 10th problem asks for a general method for solving Problem \ref{prob:main} for all Diophantine equations. It is clear that Hilbert expected a positive answer to this question. However, Davis, Putnam and Robinson \cite{davis1961decision} proved in 1961 that for Diophantine equations in which some of the exponents may be unknowns, the answer is negative, and no such general method exists. Building on this work, Matiyasevich \cite{matijasevic1970enumerable} proved in 1970 that Hilbert’s 10th problem has a negative answer for polynomial Diophantine equations as well. In other words, there is no algorithm which takes the coefficients of the polynomial $P$ as an input, runs for a finite time, and correctly outputs whether equation  \eqref{eq:diofgen} has an integer solution. Moreover, it is known that the problem remains undecidable even for some restricted families of polynomial Diophantine equations, such as equations in at most $11$ variables \cite{zhi2021further} or equations of degree at most $8$ \cite{jones1980undecidable}\footnote{If the task is to determine the existence of a \emph{positive} integer solution, then there is no algorithm for equations of degree $4$. On the other hand, if the task is to determine the existence of a rational solution to \eqref{eq:diofgen}, then even the existence of a general algorithm for all equations has not been ruled out yet \cite{sturmfels1987decidability}.}. See excellent recent surveys of Gasarch \cite{gasarch2021hilbertb, gasarch2021hilberta, gasarch2021hilberts} for more detailed discussions concerning for which families of Diophantine equations the Hilbert 10th problem is undecidable, and for which families it is known to be decidable. Also, there are explicit examples of one-parameter families of Diophantine equations for which it is undecidable to determine for which values of the parameter an equation is solvable \cite{jones1980undecidable}. However, these examples are quite complicated. 

Can we at least solve all ``simple-looking'' Diophantine equations? 
This paper is inspired by the following question, asked on the mathoverflow website \cite{Z2018}: What is the smallest Diophantine equation for which the Problem \ref{prob:main} is open? The measure $H$ of ``size'' of a Diophantine equation \eqref{eq:diofgen} suggested in \cite{Z2018} is the following one: substitute $2$ in the polynomial $P$ instead of all variables, absolute values instead of all coefficients, and evaluate. In other words, if $P$ in the reduced form has $k$ monomials of degrees $d_1, \dots, d_k$ with coefficients $a_1, \dots, a_k$, respectively, then 
\begin{equation}\label{eq:Hdef}
H(P) = \sum_{i=1}^k |a_i|2^{d_i}.
\end{equation} 
For example, for the equation
\begin{equation}\label{eq:Booker}
x^3+y^3+z^3-33=0,
\end{equation} 
we have
$$
H(x^3+y^3+z^3-33)=2^3+2^3+2^3+33 = 57.
$$

Problem \ref{prob:main} for the equation \eqref{eq:Booker} was open at the time the question \cite{Z2018} was asked, but was later solved by Andrew R. Booker \cite{booker2019cracking}. The answer turned out to be ``Yes'', and the solution Booker found is 
\begin{multline}
  8,866,128,975,287,528^3 + (-8,778,405,442,862,239)^3  
  +(-2,736,111,468,807,040)^3 = 33 
\end{multline}
To the best of our knowledge, Problem \ref{prob:main} remains open for the equation $x^3+y^3+z^3-114=0$ of size $H=138$. 

In this paper, we study the solvability of Diophantine equations systematically, starting from equations with $H=0,1,2,3,\dots$, and so on. In the process, we review some techniques for solving Diophantine equations, and list certain families of equations which are solvable by these techniques. We then present the ``smallest'' equations outside of these families, solve some of these equations, and suggest many other equations to the readers as open questions. 

We conclude the introduction by discussing why this particular choice of measure of equation ``size'' has been chosen in this paper. 
The first obvious property of $H$ is that, for any integer $B>0$, there are only finitely many Diophantine equations of size $H \leq B$. This property, which we call \emph{the finiteness property}, fails for many standard ``measures of simplicity'' of polynomials. For example, if we define the height of a polynomial as the maximum absolute value of its coefficients, then there are infinitely many polynomials with height $1$. Similarly, there are infinitely many equations with a given degree or a given number of variables.

Of course, one may define many other ``measures of equation size'' satisfying the finiteness property, for example, substitute any other constant (such as $3$) instead of $2$ in \eqref{eq:Hdef}. We next give some justification of this particular formula and the constant $2$ in it. For simplicity, let us consider monomials, and discuss how many symbols we need to write down a given monomial. If we do not use the power symbol, and write, for example, $x^3y^2$ as $xxxyy$, we need exactly $d$ symbols to write a monomial of degree $d$ and coefficient $a=1$. If $a\neq 1$, we also need about
$\log_2(|a|)$ symbols to write down $a$ in binary (ignoring sign). So, let us define the length of a monomial $M$ of degree $d$ with coefficient $a$ as
$$
l(M) := \log_2(|a|) + d.
$$    
Length $l(M)$ has the disadvantage of being not always integer, but ordering monomials by $l$ is equivalent to ordering them by 
$$
2^{l(M)} = 2^{\log_2(|a|)+d} = |a|2^d.
$$
This is exactly the size $H$ of $M$ defined in \eqref{eq:Hdef}. So, at least for monomials, $H$ has the meaning of being a monotone transformation of (the approximation of) the number of symbols needed to write $M$, and the constant $2$ corresponds to the fact that coefficients are written in binary, which is a reasonable and standard assumption.

The discussion above explains why $H$ is a natural choice, but of course does not imply that $H$ is the \emph{only} possible choice. For example, a natural alternative would be to accept the fact that length can be irrational, not transform it to integers, and define the length of a polynomial $P$ consisting of monomials of degrees $d_1, \dots, d_k$ with coefficients $a_1, \dots, a_k$ as
\begin{equation}\label{eq:ldef}
l(P) = \sum_{i=1}^k  \log_2(|a_i|) + \sum_{i=1}^k d_i = \log_2\left(\prod_{i=1}^k|a_i|\cdot 2^{\sum_{i=1}^k d_i}\right).
\end{equation}
In Section \ref{sec:short}, we will investigate what happens if we order the polynomials by $l$ instead of $H$, and conclude that we will end up of studying a similar set of equations, just arriving in a different order. The reason is that there are some equations that are hard to solve but at the same time are amazingly simple to write down, and these equations come up reasonably soon in any natural ordering of equations. In this sense, the particular choice of how to order equations (by $H$, by $l$, or in some other natural way) does not matter too much.

The contribution and organization of this work are as follows. Section \ref{sec:general} reviews some 
known methods and algorithms for determining whether an equation has integer solutions, such as the Hasse principle, prime factors analysis using the law of quadratic reciprocity, the Vieta jumping technique, etc. We then use these methods to solve, with computer assistance, Hilbert's tenth problem for all\footnote{Here, by ``all'' we mean all equations that have not been solved by other people. For equations solved by others, a reference is given.} equations of size $H\leq 30$, and also for all equations of size $H=31$ with a single exception, see equation \eqref{eq:h31main}. In Section \ref{sec:special}, we consider special classes of equations, such as equations in two variables, symmetric equations, and equations with three monomials, and, in each category, solve all equations up to a certain size. In addition, Section \ref{sec:short} solves all equations of length $l<10$, and also all equations of length $l=10$ with three exceptions. Section \ref{sec:concl} concludes the work and lists some directions for further research. 

\section{General Diophantine equations}\label{sec:general}

\subsection{$H \leq 16$: trivial and well-known equations.}\label{sec:trivial}

This section investigates Problem \ref{prob:main} for Diophantine equations \eqref{eq:diofgen} for polynomials $P$ with given $H=0,1,2,3,\dots$. We used a simple computer program for enumerating all such equations. The program returns equation $0=0$ for $H=0$, equations $\pm 1 = 0$ for $H=1$, equations $\pm 2 =0$ and $\pm x = 0$ for $H=2$, and so on. For values of $H\leq 14$, all the equations are uninteresting and belong to at least one of the following families of equations with trivially solvable Problem \ref{prob:main}.

\begin{itemize}
\item \textbf{Equations with no variables} like $0=0$, $\pm 1=0$, $\pm 2=0$, etc.
\item \textbf{Equations with small solutions}, e.g. equation $xy+y+1=0$ has solution $x=0, y=-1$. 
In this case, the answer to Problem \ref{prob:main} is trivially ``Yes'', and such equations can be safely excluded from the analysis. Specifically, we have excluded all equations that have a solution with $\max_i{|x_i|}\leq 100$. In particular, this excludes equations with no free term, for which $P(0,\dots,0)=0$. In fact, after multiplying by $-1$ if necessary, we may assume that $P(0,\dots,0)>0$. After this modification, the program returns no equations of size $H\leq 4$ but returns equations $x^2+1=0$ and $\pm 2x+1=0$ of size $H=5$.
\item \textbf{Equations in one variable} 
\begin{equation}\label{eq:onevar}
a_m x^m + \dots + a_1 x + a_0 = 0
\end{equation}
with $a_0\neq 0$, in which any integer solution must be a divisor of $a_0$. In fact, all integer solutions to \eqref{eq:onevar} can be listed in time polynomial in the size of the input \cite{cucker1999polynomial}. We added a condition to the code to exclude all such equations, which immediately allowed us to exclude all the equations up to $H\leq 8$. More generally, this is the way we proceed further in this paper: identify a class of equations for which Problem \ref{prob:main} can be solved, modify the code to exclude all equations from this class from further consideration, find the equations with smallest $H$ not excluded so far, and then repeat the process.
 
\item \textbf{Equivalent equations.} The program next returns many equations of size $H=9$, including, for example, four equations $\pm 2x \pm 2y+1=0$. As a next step, we call equations equivalent if they can be transformed to each other by substitutions $x_i \to -x_i$ and/or permutations of variables, and exclude all equations except one from every such equivalence class. For example, the program then returns only one equation $2x+2y+1=0$ from the aforementioned four.  

\item We next exclude all \textbf{linear equations}
$
a_1 x_1 + a_2 x_2 + \dots + a_n x_n + b = 0,
$
because they all can be easily solved, in fact in polynomial time \cite{chou1982algorithms}. More generally, we may exclude all equations of the form 
\begin{equation}\label{eq:lingen}
a x_1 + Q(x_2,\dots,x_n) = 0
\end{equation}
where $a\neq 0$ is an integer and $Q$ is a polynomial with integer coefficients. For such equations, Problem \ref{prob:main} has a ``Yes'' answer if and only if $Q(x_2,\dots,x_n)$ is a multiple of $a$ for some integers $x_2,\dots,x_n$. Note that \eqref{eq:lingen} covers all linear equations and also some non-linear ones like $3y=x^2+1$.

\item \textbf{Equations that have no real solutions}. Obviously, such equations have no integer solutions either. A simple example is the equation $x^2+y^2+1=0$ of size $H=9$. All equations with no real solutions are recognizable in finite time \cite{tarski1998decision}.  
Moreover, we have also excluded equations for which the inequality $P(x_1, \dots, x_n)\leq 0$ (or $P(x_1, \dots, x_n)\geq 0$) has a finite number of integer solutions, which also allows to solve Problem \ref{prob:main} in finite time. A simple example is the equation $3-x^2-y^2=0$ of size $H=11$.

\item \textbf{Equations with no solutions modulo $a$}. We next exclude equations $P(x_1,\dots,x_n)=0$ for which there exists an integer $a\geq 2$ such that $P(x_1,\dots,x_n)$ is not a multiple of $a$ for any integers $x_1,\dots,x_n$. Obviously, 
then the equation has no integer solutions. A simple example is the equation $2xy+1=0$ of size $H=9$, where the left-hand side is always odd. The divisibility condition can be easily checked for any fixed $a$ by enumerating all possible remainders that $x_1,\dots,x_n$ can give after division by $a$. Moreover there is an algorithm developed by Ax \cite{ax1967solving} in 1967 that checks this condition for all values of $a$ in finite time\footnote{The main theorem in \cite{ax1967solving} gives an algorithm for solving a Diophantine equation modulo every prime $p$. However, on page 11 of \cite{ax1967solving}, it is remarked that ``This gives an algorithm for determining whether a given system of Diophantine equations has, for all primes $p$, a solution over $Z_p$'', where $Z_p$ is the set of p-adic integers. This is equivalent 
to the solvability of the equation modulo every prime power, which in turn is equivalent to its solvability modulo every integer $a$ by the Chinese remainder theorem.}. This allows us to exclude all equations of size $H\leq 12$.
 
\item Equation $(x^2+2)y=1$ of size $H=13$ is the smallest one which has real solutions and also solutions modulo every integer, but still has no integer solutions, because $\frac{1}{x^2+2}$ is never an integer. We next exclude all \textbf{equations in the form} 
\begin{equation}\label{eq:polprod}
P_1 \cdot P_2 \cdot ... \cdot P_k - a = 0,
\end{equation}
where $P_i$ are polynomials and $a$ is a constant, for which there is no factorisation $a=u_1\cdot ... \cdot u_k$ of $a$ such that the system of equations $P_i=u_i, \, i=1,\dots,k$ is solvable in integers. For all such equations, the answer to Problem \ref{prob:main} is ``No''. This allows us to exclude all equations of size $H\leq 14$.
\item \textbf{Equations that can be reduced to equations with smaller $H$}. If polynomial $P$ in the equation $P=0$ is representable as a product of polynomials with smaller $H$, the resulting equation
\begin{equation}\label{eq:product}
\prod\limits_{i=1}^k P_i(x_1,\dots,x_n) = 0
\end{equation}
has an integer solution if and only if at least one of the equations $P_i(x_1,\dots,x_n) = 0, \, i=1,\dots,k$ has an integer solution. 
In some other cases, an equation can be reduced to a smaller one by a linear substitution, e.g. equation $(y+1)^2=x^3-3$ of size $H=20$ reduces to $z^2=x^3-3$ of size $H=15$ after linear substitution $y+1 \to z$. Another example of an equation that trivially reduces to $z^2=x^3-3$ is the equation $(yz)^2=x^3-3$. In general, if an equation $A$ can be reduced to a smaller equation $B$, it does not mean that $A$ is trivial, because $B$ may be far from trivial itself. However, because we study all equations systematically in order, any such reduction implies that equation $B$ has been already considered, hence equation $A$ can be safely excluded.
\end{itemize}

For all the equations listed above Problem \ref{prob:main} is trivial and we have excluded all such equations from consideration. This automatically excluded all equations of size $H\leq 14$, and the smallest equation the program classified as non-trivial was the equation 
\begin{equation}\label{eq:h15}
y^2=x^3-3
\end{equation}
of size $H=15$.

As a next step, we have also excluded some classes of equations for which Problem \ref{prob:main} can be highly non-trivial but is known to be solvable in finite time. Specifically, we have excluded:

\begin{itemize}

\item \textbf{Quadratic equations}. A deep theorem of Grunewald and Segal \cite{grunewald1981solve} proves the existence of a finite algorithm for solving Problem \ref{prob:main} for quadratic equations in any number of variables.
\begin{theorem}\label{th:quadratic}
There is an algorithm which, given any quadratic polynomial 
$$
P(x_1, \dots, x_n) = \sum_{i=1}^n \sum_{j=1}^n a_{ij} x_i x_j + \sum_{i=1}^n b_i x_i + c
$$
with integer coefficients, determines whether equation $P=0$ has an integer solution.
\end{theorem}
For $n=2$ variables, there is an implementation available online \cite{Alpern} that not only decides whether a quadratic equation has an integer solution, but in fact determines all integer solutions. If there are infinitely many solutions, they are described in the form of expressions with parameters or recurrence relations.
For our purposes, it suffices that the algorithm checks whether the equation has a finite number of integer solutions, and if so, lists all the solutions.

Theorem \ref{th:quadratic} allows us to exclude all quadratic equations.

We remark that excluding quadratic equations also allows us to exclude all $n$-variable equations of size $H \leq 4n+4$.

\begin{proposition}\label{prop:4nplus4}
Let $P(x_1, \dots, x_n)$ be a polynomial in $n$ variables with integer coefficients such that either (a) $H(P)\leq 4n$ or (b) $P$ has degree at least $3$ and $H(P)\leq 4n+4$. Then equation \eqref{eq:diofgen} has an integer solution.
\end{proposition}
\begin{proof}
Let us write down each monomial of $P$ without coefficients and power symbols, for example, write $2x^2y$ as $xxy+xxy$, etc. We call this the ``expanded form'' of $P$. If there is a variable $x_i$ that is present in the expanded form only once, then setting all other variables to $1$ and calculating $x_i$ we obtain an integer solution to \eqref{eq:diofgen}. Hence, it is left to consider the case when each $x_i$ appears at least twice. Now, to calculate $H$, replace all ``$-$'' by ``$+$'' and all $x_i$ by $2$, and get that
$$
H = |a| + E,
$$
where $a$ is the free term of $P$, and $E$ is the expression containing $k\geq 2n$ two-s and operations ``$+$'' and ``$\cdot$''. Because replacing ``$\cdot$'' by ``$+$'' can only decrease $E$, we have $E \geq 2k \geq 4n$. Hence, if $H \leq 4n$ then $|a|=0$, and \eqref{eq:diofgen} has solution $x_1=\dots=x_n=0$, which proves (a).

Now assume that $P$ has degree at least $3$ and $H\leq 4n+4$. If $E=4n+4$ then $|a|=0$ and \eqref{eq:diofgen} is solvable, hence we may assume $E<4n+4$. Because $E$ is even, it follows that $E\leq 4n+2$. Because $P$ has at least one term of degree at least $3$, and replacing $2\cdot 2\cdot 2$ by $2+2+2$ decreases $E$ by $2$, we must have $E\geq 2k+2=4n+2$. Hence, equality holds, which means that $P$ has exactly one cubic term with coefficient $\pm 1$, and each variable $x_i$ appears in the extended form of $P$ exactly twice. This means that the cubic term cannot have the form $x_i^3$, so it must be either $\pm x_i^2x_j$ or $\pm x_ix_jx_u$. Then the variable $x_j$ must enter some other monomial, either $\pm x_j$ or $\pm x_jx_v$, and, by exchanging $i$ and $u$ if necessary, we may assume that $v\neq i$. But then taking $x_i=0$, $x_v=1$, other variables arbitrary, and estimating $x_j$, we may find an integer solution to \eqref{eq:diofgen}.  
\end{proof}

The bounds on $H$ in Proposition \ref{prop:4nplus4} are the best possible for every $n\geq 2$, as witnessed by equations (a) $2\sum_{i=1}^nx_i+1=0$ (or $\sum_{i=1}^nx_i^2+1=0$) of size $H=4n+1$ and (b) $(x_1^2+x_1)x_2+2\sum_{i=3}^nx_i+1=0$ of size $H=4n+5$ that have no integer solutions.

Theorem \ref{th:quadratic} and Proposition \ref{prop:4nplus4} imply that for $H\leq 16$ it suffices to consider only cubic equations in $2$ variables. We discuss such equations next. 

\item \textbf{Cubic equations in two variables}, the smallest non-trivial example is \eqref{eq:h15}. 
This is the smallest equation whose solvability problem 
requires at least some thinking and/or at least minimal background in number theory. 
However, this equation belongs to a family for which Problem \ref{prob:main} is known to be solvable. Indeed, we have the following well-known results.

\begin{theorem}\label{th:notabsirr}
Let $P(x,y)$ be a polynomial with integer coefficients which is irreducible over ${\mathbb Q}$ but not absolutely irreducible\footnote{A polynomial $P$ with integer coefficients is called absolutely irreducible if it cannot be written as a product $P=P_1\cdot P_2$ of non-constant polynomials, even if we allow complex coefficients.}. Then there is an algorithm for determining all integer solutions (in fact, all rational solutions) to the equation $P(x,y)=0$. 
\end{theorem}
\begin{proof}
For every polynomial $P(x,y)$ there exists a unique homogeneous polynomial $Q(x,y,z)$ of the same degree such that $P(x,y)=Q(x,y,1)$. Solutions to the system of equations
$$
\frac{\partial Q}{\partial x} = \frac{\partial Q}{\partial y} = \frac{\partial Q}{\partial z} = 0
$$
are called singular points. If $P(x,y)$ is irreducible over ${\mathbb Q}$ but not absolutely irreducible, all rational solutions to $P(x,y)=0$ correspond to singular points. There is a finite number of singular points and an algorithm \cite{sakkalis1990singular} that can list them all. Then we can check which of these singular points lead to a rational (or integer) solutions to $P(x,y)=0$.
\end{proof}

\begin{theorem}\label{th:cubic2var}
There is an algorithm that, given a polynomial $P(x,y)$ of degree at most $3$ with integer coefficients, determines all integer solutions to the equation
\begin{equation}\label{eq:cubic2var}
P(x,y)=0.
\end{equation}
In particular, Problem \ref{prob:main} is solvable for this class of equations.
\end{theorem}
\begin{proof}
A combination of the main theorems in \cite{baker1970integer}, \cite{poulakis1993points}, and \cite{poulakis2002solving} imply the result in the case when $P(x,y)$ is absolutely irreducible. If $P(x,y)$ is irreducible over ${\mathbb Q}$ but not absolutely irreducible, then the result follows from Theorem \ref{th:notabsirr}. Finally, if $P(x,y)$ is reducible over ${\mathbb Q}$, it can be written as a product of a linear polynomial $P_1(x,y)$ and a quadratic polynomial $P_2(x,y)$ with rational coefficients. Then the equations $P_i(x,y)=0$, $i=1,2$ can be solved by known algorithms \cite{chou1982algorithms, Alpern}.
\end{proof}

While the original algorithm in \cite{baker1970integer} takes a finite but long time to run, subsequent authors \cite{pethHo1999computing, stroeker2003computing} developed alternative practical algorithms.
For example, there is an implementation in SageMath, which is an open-source and free-to-use mathematical software system \cite{zimmermann2018computational}, of an algorithm for finding all integer solutions to the equations in the form
\begin{equation}\label{eq:ellWei}
y^2 + a x y + c y = x^3 + b x ^2 + d x + e
\end{equation} 
under some minor conditions on (integer) coefficients $a,b,c,d,e$. Such equations are known as elliptic curves in Weierstrass form, and can be solved by a single command
\begin{equation}\label{eq:commandellWei}
sage: EllipticCurve([a,b,c,d,e]).integral\_points()
\end{equation}
that can be run online at \url{https://sagecell.sagemath.org/}. For example, the command
$$
sage: EllipticCurve([0,0,0,0,-3]).integral\_points()
$$
Returns the empty set, which means that equation \eqref{eq:h15} has no integer solutions. 
In fact, all $2$-variable cubic equations that our program returns up to $H\leq 18$ happened to be in the form \eqref{eq:ellWei} and can be solved by the command \eqref{eq:commandellWei}. The first exception is the equation
$$
2y^2 = x^3 - 3
$$
of size $H=19$, but it can be reduced to $Y^2=X^3-24$ after multiplying by $8$ and using substitution $X=2x$, $Y=4y$. Some other $2$-variable cubics, like, for example,
$$
x^3+x^2y-y^3-y+3 = 0,
$$
are more interesting, but all of them are covered by Theorem \ref{th:cubic2var}, see also 
\cite{stroeker1999solving} for a more practical method. All further equations covered by Theorems \ref{th:notabsirr} and \ref{th:cubic2var} will be excluded from further analysis.
\end{itemize}

If an equation has $H\leq 16$, then it is either quadratic and covered by Theorem \ref{th:quadratic}, or cubic in two variables and covered by Theorem \ref{th:cubic2var}, or cubic in $n\geq 3$ variables and covered by Proposition \ref{prop:4nplus4}. Hence, all equations of size $H\leq 16$ are either trivial or belong to some family for which a finite algorithm for Problem \ref{prob:main} is well-known.

\subsection{$H \geq 17$: prime factors of quadratic forms}\label{sec:residues}

The smallest equation which is not excluded by the criteria described in Section \ref{sec:trivial} is the equation 
\begin{equation}\label{eq:h17}
y(x^2-y)=z^2+1
\end{equation}
of size $H=17$. 
In fact, 
the only equations the program returns for $H\leq 20$ are \eqref{eq:h17} and a similar equation
\begin{equation}\label{eq:h19}
y(x^2+3) = z^2+1,
\end{equation} 
of size $H=19$.

These equations are not completely trivial, but they are also not difficult. The key idea is to understand what prime numbers can be divisors of $z^2+1$. For example, it is easy to see that $z^2+1$ is never divisible by $3$, because if $z \equiv 0,1, \text{ or }2 (\text{mod } 3)$, then $z^2+1 \equiv 1,2, \text{ or }2 (\text{mod } 3)$, respectively\footnote{For integer $m$, we will write $a \equiv b (\text{mod } m)$ if $a-b$ is a multiple of $m$.}. More generally, we have the following well-known fact.

\begin{proposition}\label{prop:factorsz2p1}
For any integer $z$, all odd prime factors $p$ of $z^2+1$ are in the form $p=4k+1$.
\end{proposition}
\begin{proof}
Let $p$ be any odd prime divisor of $z^2+1$. Then $z^2\equiv-1(\text{mod } p)$, which implies that $z^{p-1}=(z^2)^{\frac{p-1}{2}} \equiv (-1)^{\frac{p-1}{2}}(\text{mod } p)$. But Fermat's little theorem implies that $z^{p-1} \equiv 1 (\text{mod } p)$. Hence $1 \equiv (-1)^{\frac{p-1}{2}} (\text{mod } p)$. Because $1 \not\equiv -1 (\text{mod } p)$, this implies that $\frac{p-1}{2}$ is even, hence $p=4k+1$ for some integer $k$. 
\end{proof}

We now apply Proposition \ref{prop:factorsz2p1} to solve equations \eqref{eq:h17} and \eqref{eq:h19}. We remark that equation \eqref{eq:h17} was first solved by Victor Ostrik in the comments section to the mathoverflow question \cite{Z2018}.

\begin{proposition}\label{prop:h17to20}
Equations \eqref{eq:h17} and \eqref{eq:h19} have no integer solutions.
\end{proposition}
\begin{proof}
Because $z^2 \not\equiv 3 (\text{mod } 4)$, $z^2+1$ is not divisible by $4$. Hence, the prime factorization of $z^2+1$ is $2^{s}\prod_i p_i$, where $s$ is $0$ or $1$, and all $p_i$ are $1$ modulo $4$ by Proposition \ref{prop:factorsz2p1}. This implies that all positive divisors of $z^2+1$ are $1$ or $2$ modulo $4$, and if divisor $d$ is even, then $(z^2+1)/d$ is odd.   

Because $x^2+3$ is positive and is never $1$ or $2$ modulo $4$, this implies that equation \eqref{eq:h19} has no integer solutions. In equation \eqref{eq:h17}, $d_1=y$ and $d_2=x^2-y$ are divisors of $z^2+1$. Because $d_1 d_2 = z^2+1 > 0$ and $d_1 + d_2 = x^2 \geq 0$, both $d_1$ and $d_2$ are positive. Hence, they are both equal to $1$ or $2$ modulo $4$, with at most one $d_1$ or $d_2$ even. But then $x^2=d_1+d_2$ is $2$ or $3$ modulo $4$, which is a contradiction.
\end{proof}

Proposition \ref{prop:h17to20} finishes the analysis of all equations of size $H\leq 20$. For $H=21$, the only equation the program returns is 
\begin{equation}\label{eq:h21}
y(x^2+2) = 2z^2-1.
\end{equation} 
%
To solve this equation by the same method, we need to understand the possible prime factors of $2z^2-1$. Let us discuss, more generally, the prime factors of an arbitrary quadratic form $ax^2+bxy+cy^2$. This can be done using the theory of quadratic residues. 
Given an odd prime $p$ and integer $a$ with $\text{gcd}(a,p)=1$, if there exists an integer $z$ such that $z^2 \equiv a (\text{mod } p)$, we say that $a$ is a quadratic residue modulo $p$ and write $\left(\frac{a}{p}\right)=1$, otherwise we say that $a$ is a quadratic non-residue modulo $p$ and write $\left(\frac{a}{p}\right)=-1$. In this notation, Proposition \ref{prop:factorsz2p1} states that if $p$ is an odd prime with $\left(\frac{-1}{p}\right) = 1$ then $p \equiv 1 (\text{mod } 4)$. In fact, the converse direction is also true, and this can be written in a compact form as
\begin{equation}\label{eq:suppl1}
\left(\frac{-1}{p}\right) = (-1)^{\frac{p-1}{2}}.
\end{equation} 
Moreover, we have the formula
$
\left(\frac{2}{p}\right) = (-1)^{\frac{p^2-1}{8}},
$
the multiplicative law
$
\left(\frac{a b}{p}\right) = \left(\frac{a}{p}\right) \left(\frac{b}{p}\right),
$
and, most importantly, the law of quadratic reciprocity, which states that if $p$ and $q$ are distinct odd prime numbers, then
\begin{equation}\label{eq:repos}
\left(\frac{p}{q}\right) \left(\frac{q}{p}\right) = (-1)^{\frac{p-1}{2}\frac{q-1}{2}}.
\end{equation} 
These formulae suffice to easily compute $\left(\frac{a}{p}\right)$ in general. For example,
for odd $p\neq 3$, \eqref{eq:repos} implies that $\left(\frac{p}{3}\right) \left(\frac{3}{p}\right) = (-1)^{\frac{p-1}{2}\frac{3-1}{2}}=(-1)^{\frac{p-1}{2}}$, hence
\begin{equation}\label{eq:repos3}
\left(\frac{3}{p}\right) = 1 \quad \Leftrightarrow \quad p \equiv 1 \text{ or } 11 (\text{mod } 12).
\end{equation}

The prime factors of a quadratic form are characterised by the following well-known fact.

\begin{proposition}\label{prop:quadform}
Let $a,b,c$ be integers such that $\text{gcd}(a,b,c)=1$, and let $D=b^2-4ac$. Let $p$ be a prime factor of $ax^2+bxy+cy^2$ for some integers $x,y$. Then either $p$ is a divisor of $2D$, or $\left(\frac{D}{p}\right)=1$, or $p$ is a common divisor of $x$ and $y$.
\end{proposition}
\begin{proof}
Assume that an odd prime $p$ is a divisor of $ax^2+bxy+cy^2$, but not a divisor of $D$ and not a common divisor of $x$ and $y$. By symmetry, we may assume that $p$ is not a divisor of $y$, so that $\text{gcd}(y,p)=1$. Then there exists an integer $z$ such that $yz \equiv 1 (\text{mod } p)$. As $p$ is a divisor of $ax^2+bxy+cy^2$, we have 
$$
0 \equiv 4az^2(ax^2+bxy+cy^2) = z^2(2ax+by)^2 - D(yz)^2 \equiv (2axz+byz)^2 - D (\text{mod } p).
$$
Hence, $\left(\frac{D}{p}\right)=1$. 
\end{proof}

Let us apply Proposition \ref{prop:quadform} to some specific quadratic forms.

\begin{corollary}\label{cor:quadform}
Let $x,y$ be integers and let $p$ be an odd prime that is not a common divisor of $x$ and $y$. Then:
\begin{center}
\begin{tabular}{ |c|c|c|c|c|c|c| } 
 \hline
 if $p$ is a prime factor of & then \\ 
 \hline
 $x^2+y^2$ & $p \equiv 1 (\text{mod } 4)$ \\ 
 \hline
 $x^2+2y^2$ & $p \equiv 1 \text{ or } 3 (\text{mod } 8)$ \\ 
 \hline
 $x^2-2y^2$ & $p \equiv 1 \text{ or } 7 (\text{mod } 8)$ \\
 \hline
 $x^2+3y^2$ & $p=3$ or $p \equiv 1 (\text{mod } 3)$ \\  
 \hline
 $x^2-3y^2$ & $p=3$ or $p \equiv 1 \text{ or } 11 (\text{mod } 12)$ \\ 
 \hline
\end{tabular}
\end{center}
\end{corollary}
\begin{proof}
Let us prove for example the statement about the prime divisors of $x^2-3y^2$. The proofs of all other statements are similar. We will apply Proposition \ref{prop:quadform} with $a=1$, $b=0$, $c=-3$, hence $D=b^2-4ac=12$. If $p\neq 3$ is an odd prime, it is not a divisor of $2D=24$. Hence, if $p$ is a divisor of $x^2-3y^2$ and $\text{gcd}(p,x,y)=1$, Proposition \ref{prop:quadform} implies that $\left(\frac{12}{p}\right)=1$. Because $\left(\frac{12}{p}\right)=\left(\frac{3}{p}\right)\left(\frac{2}{p}\right)^2=\left(\frac{3}{p}\right)$, this is equivalent to $\left(\frac{3}{p}\right)=1$, and \eqref{eq:repos3} implies that $p \equiv 1 \text{ or } 11 (\text{mod } 12)$.
\end{proof}

We will now use Corollary \ref{cor:quadform} to solve equation \eqref{eq:h21}.

\begin{proposition}\label{prop:h21}
Equation \eqref{eq:h21} has no integer solutions.
\end{proposition}
\begin{proof}
Let $p$ be any prime factor of $2z^2-1$. Because $2z^2-1$ is odd, $p\neq 2$.  Because $2z^2-1=-(1^2-2z^2)$, Corollary \ref{cor:quadform} implies that $p \equiv 1 \text{ or } 7 (\text{mod } 8)$. All positive divisors of $2z^2-1$ are the products of such primes and are therefore also $1$ or $7$ modulo $8$. Because $x^2+2$ is positive and is never $1$ or $7$ modulo $8$, it cannot be a divisor of $2z^2-1$, and equation \eqref{eq:h21} has no integer solutions.
\end{proof}

%

Proposition \ref{prop:h21} finishes the analysis of the equations of size $H \leq 21$. 



For higher $H$, the number of equations solvable by this method increases, and it is useful to automate the described method. We next present it as an (informal) algorithm that, for each equation, either proves that it has no integer solutions, or gives up.

\begin{algorithm}\label{alg:quadform}~ 
Assume that the equation is presented in the form
\begin{equation}\label{eq:prodform}
\prod_{j=1}^k P_j = Q,
\end{equation}
where $P_1,\dots,P_k,Q$ are non-constant polynomials in variables $x_1,\dots,x_n$ with integer coefficients. Then return that the equation is ``Solved'' if it is  possible to use Proposition \ref{prop:quadform} to find integers $m\geq 3$ and $0\leq r_1 < \dots < r_l < m$ such that 
\begin{itemize}
\item[(a)] for any integers $x_1,\dots,x_n$, all positive divisors of $Q(x_1,\dots,x_n)$ must be equal to some $r_j$ modulo $m$, but
\item[(b)] there is no solution $x_1,\dots,x_n$ of \eqref{eq:prodform} modulo $m$ such that all $|P_j(x_1,\dots,x_n)|$ are equal to some $r_j$ modulo $m$.
\end{itemize}
\end{algorithm}

Given an equation $P=0$, there are infinitely many way to represent it in the form \eqref{eq:prodform}, and, in general, it is non-trivial to find a representation for which Algorithm \ref{alg:quadform} works. We start with the following trivial method, which already suffices to solve many equations.

\begin{algorithm}\label{alg:quadtriv}~ 
The input is polynomial Diophantine equation $P(x_1,\dots,x_n)=0$. For each $1\leq i \leq n$ let $Q_i=P(x_1,\dots,x_{i-1}, 0, x_{i+1}, \dots, x_n)$ be the polynomial resulting from substituting $x_i=0$ into $P$. Then polynomial $P-Q_i$ is reducible over ${\mathbb Q}$, because at least $x_i$ is a non-trivial factor. Represent $P-Q_i = \prod_{j=1}^k P_j$ as a product of irreducible factors, and then run Algorithm \ref{alg:quadform} for the representation 
\begin{equation}\label{eq:quadtriv}
\prod_{j=1}^k P_j = - Q_i.
\end{equation}
If the algorithm returns ``Solved'' for some $i$, stop and report the equation as solved.
\end{algorithm}

Table \ref{tab:H25} enumerates all equations up to size $H \leq 25$ that are not excluded by the criteria described in Section \ref{sec:trivial} but are solvable by Algorithm \ref{alg:quadtriv}. For each equation, we give the representation \eqref{eq:quadtriv} that worked, and the integers $m$ and $r_1,\dots,r_l$ returned by Algorithm \ref{alg:quadform}. Each equation in the table has no integer solutions, and we can write a unified human-readable proof of this fact as follows: The equation (insert equation from Column 2) can be rewritten in the form \eqref{eq:quadtriv} as (insert representation from Column 3). Then Proposition \ref{prop:quadform} implies that the right-hand side can only have positive divisors equal to $r_1,\dots,r_l$ modulo $m$, but the equation has no solution modulo $m$ for which all the factors of the left-hand side satisfy this condition. 

Starting with $H\geq 26$, equations solvable by Algorithm \ref{alg:quadtriv} will be excluded from the analysis and not listed.

\begin{table}
\begin{center}
\begin{tabular}{ |c|c|c|c|c| }
 \hline
 $H$ & Equation & Representation \eqref{eq:quadtriv} & $m$ & $r_1,\dots,r_l$  \\ 
 \hline\hline
 $17$ & $1+x^2 y+y^2+z^2=0$ & $y(x^2+y)=-1-z^2$ & $4$ & $1,2$ \\ 
 \hline
 $19$ & $1+3 y+x^2 y+z^2=0$ & $(3+x^2) y=-1-z^2$ & $4$ & $1,2$ \\ 
 \hline
 $21$ & $1+2 y+x^2 y-2 z^2=0$ & $(2+x^2) y=-1+2 z^2$ & $8$ & $1,7$ \\ 
 \hline
 $22$ & $2+4 y+x^2 y+z^2=0$ & $(4+x^2) y=-2-z^2$ & $8$ & $1,2,3,6$ \\ 
 \hline
 $22$ & $2+4 y+x^2 y-z^2=0$ & $(4+x^2) y=-2+z^2$ & $8$ & $1,2,6,7$ \\ 
 \hline
 $22$ & $2-4 y+x^2 y-z^2=0$ & $(-2+x)(2+x)y=-2+z^2$ & $8$ & $1,2,3,6$ \\ 
 \hline
 $22$ & $4-y+x^2 y+y^2+z^2=0$ & $y(-1+x^2+y)=-4-z^2$ & $16$ & $1,2,4,5,8,9,10,13$ \\ 
 \hline
 $25$ & $1+x^2+x^2 y+y^2+y z^2=0$ & $y (x^2+y+z^2)=-1-x^2$ & $4$ & $1,2$ \\ 
 \hline
 $25$ & $1+4 y+x^2 y+2 z^2=0$ & $(4+x^2) y=-1-2 z^2$ & $8$ & $1,3$ \\ 
 \hline
 $25$ & $1+4 y+x^2 y-2 z^2=0$ & $(4+x^2) y=-1+2 z^2$ & $8$ & $1,7$ \\ 
 \hline
 $25$ & $1-4 y+x^2 y-2 z^2=0$ & $(-2+x)(2+x)y=-1+2 z^2$ & $8$ & $1,7$ \\ 
 \hline
 $25$ & $1+2 y+x^2 y+2 z+2z^2=0$ & $(2+x^2) y=-1-2 z-2z^2$ & $4$ & $1$ \\ 
 \hline
 $25$ & $1+4 y+x^2 y+y^2+z^2=0$ & $y(4+x^2+y)=-1-z^2$ & $4$ & $1,2$ \\ 
 \hline
 $25$ & $1+2 x+x^3+x y^2+z^2=0$ & $x (2+x^2+y^2)=-1-z^2$ & $4$ & $1,2$ \\ 
 \hline
\end{tabular}
\caption{\label{tab:H25} Equations of size $H\leq 25$ solvable by Algorithm \ref{alg:quadtriv}.}
\end{center} 
\end{table}

\subsection{$H \geq 22$: Markoff-type equations and Vieta jumping}

The smallest equation that is not excluded by the criteria of Section \ref{sec:trivial} and is not solvable by Algorithm \ref{alg:quadtriv} is the equation
\begin{equation}\label{eq:h22mark}
x^2 + y^2 - z^2 = xyz - 2
\end{equation} 
of size $H=22$. Equations of the form $a x^2 + b y^2 + c z^2 = d xyz + e$ are sometimes called Markoff-type equations because they generalize the famous Markoff equation $x^2 + y^2 + z^2 = 3xyz$. One of the standard approaches for solving such equations is the technique of Vieta jumping and infinite descent. One assumes the existence of a solution, and then uses this solution to produce another solution, which is strictly ``smaller'' in a certain sense. The process is then repeated to get a contradiction. We next illustrate this technique by presenting the solution of equation \eqref{eq:h22mark}. 


\begin{proposition}\label{prop:h22}
Equation \eqref{eq:h22mark} has no integer solutions.
\end{proposition}
\begin{proof}
We first present a human-readable proof, developed by Will Sawin and Fedor Petrov \cite{B2021a}. 
By contradiction, assume that an integer solution to \eqref{eq:h22mark} exists, and choose some specific solution $(x,y,z)$. Let us replace $x$ by a new variable $t$, to get a quadratic equation in $t$
$$
t^2 - yz \cdot t + y^2 - z^2 +2 = 0.
$$
Because $(x,y,z)$ is a solution to \eqref{eq:h22mark}, one root of this quadratic equation is $t_1=x$. By the properties of quadratic equations, the second root $t_2$ satisfies $t_1+t_2=yz$, or $t_2=yz-t_1=yz-x$. In other words, if $(x,y,z)$ is a solution to \eqref{eq:h22mark}, then $(yz-x,y,z)$ is also a solution. 

Now, fix any solution $(x,y,z)$ with $|x|+|y|+|z|$ minimal. First, assume that $|x|=\max(|x|,|y|,|z|)$.   Because $|x|+|y|+|z|$ is minimal, we must have $|x|\leq |yz - x|$. But then
$$
x^2 \leq |yz - x| |x| = |y^2+2-z^2|,
$$
where the equality follows from \eqref{eq:h22mark}. Because $|x|=\max(|x|,|y|,|z|)$, the last inequality is possible only if $z^2 \leq 2$. However, \eqref{eq:h22mark} with $z=0$ and $z=\pm 1$ reduces to equations $x^2 + y^2 + 2 = 0$ and $x^2 \mp xy + y^2 + 1 = 0$, and none of them has integer solutions. 

The case $|y|=\max(|x|,|y|,|z|)$ can be excluded similarly, hence it must be $|z|=\max(|x|,|y|,|z|)$. If $|z|=|x|$, then $z^2=x^2$, and \eqref{eq:h22mark} implies that $y$ is a divisor of $2$, but trying all divisors we find no solutions, a contradiction. The case $|z|=|y|$ is excluded similarly, hence $|z|>\max(|x|,|y|)$. By changing signs, we may assume that $x,y>0$. If $z\geq 0$, then $z>\max(x,y)$, and
$$
z^2+xyz \geq (x+1)^2+y^2 > x^2+y^2+2,
$$
which is a contradiction with \eqref{eq:h22mark}. Hence, we must have $z<0$. In this case, note that $(x,y,-xy-z)$ is another solution to \eqref{eq:h22mark}. 
Because $|x|+|y|+|z|$ is minimal, we must have $|-xy-z|\geq |z|$. But $|-xy-z|=-xy-z$, and $|z|=-z$, and inequality $-xy-z \geq -z$ is impossible if $x,y>0$.

We next present a shorter, computer-assisted proof. If $(x,y,z)$ is a solution to \eqref{eq:h22mark} with $|x+|y|+|z|$ minimal, then we must have $|yz-x|\geq |x|$, $|xz-y|\geq |y|$ and $|-xy-z|\geq |z|$. Now consider optimization problem
\begin{equation}\label{eq:h22opt}
\begin{split}
& \max\limits_{(x,y,z,t)\in {\mathbb R}^4} t \quad \text{subject to} \quad x^2 + y^2 - z^2 = xyz - 2, \\
& |x|\geq t, \, |y|\geq t, \, |z|\geq t, \, |yz-x|\geq |x|, \, |xz-y|\geq |y|, \, |-xy-z|\geq |z|. 
\end{split}
\end{equation} 
Any computer algebra system can easily solve this problem and output that the optimal value $t^*$ is finite. This means that any solution $(x,y,z)$ with $|x|+|y|+|z|$ minimal must have $|x|\leq t^*$, or $|y|\leq t^*$, or $|z|\leq t^*$. But there is only a finite number of integers $t$ with $|t|\leq t^*$, and, for each specific $t$, it is easy to check that \eqref{eq:h22mark} has no integer solutions with $x=t$ (or with $y=t$, or with $z=t$).

In our specific case the solution to problem \eqref{eq:h22opt} is in fact  $t^*=0$, so it is left to check that the equation has no integer solutions with $x=0$, or $y=0$, or $z=0$, which is straightforward.
\end{proof} 


The second proof of Proposition \ref{prop:h22} can be easily turned into a general algorithm.

\begin{algorithm}\label{alg:vieta}~ 
\begin{itemize}
\item Call variable $x_i$ of a general equation $P(x_1,\dots,x_n)=0$ ``special'' if the equation can be written in the form
$$
a_i x_i^2 + Q_i x_i + R_i = 0
$$
where $|a_i|=1$ and $Q_i$ and $R_i$ are polynomials in other variables. 
\item Solve optimization problem of maximizing $t$ over $(x_1,\dots,x_n,t)\in {\mathbb R}^{n+1}$ subject to constraints $P=0$, $|x_i|\geq t$ for each $i$, and $|-(Q_i/a_i)-x_i|\geq |x_i|$ for each special variable $x_i$. Let $t^*$ be the optimal value.
\item If $t^*=\infty$, give up. Otherwise for each integer $t$ such that $|t|\leq t^*$ and each $i=1,\dots,n$, try to prove that the original equation has no integer solutions with $x_i=t$. If this is difficult, give up, but if succeed, this implies that the original equation has no integer solutions.
\end{itemize}
\end{algorithm}

Table \ref{tab:vieta} lists equations up to $H \leq 30$ that are not excluded earlier but are solvable by Algorithm \ref{alg:vieta}. For each equation, we list the value $t^*$ returned by the Algorithm. Each equation in the table has no integer solutions, and the proof of this fact can be written in a unified way as follows: By contradiction, assume that integer solution exists. Let us choose a solution with $\sum |x_i|$ minimal. Then, in the notation of Algorithm \ref{alg:vieta}, we have $|-(Q_i/a_i)-x_i|\geq |x_i|$ for each special variable $x_i$, hence we must have $|x_i|\leq t^*$ for some $i$. However, a direct verification shows that the equation has no integer solution satisfying the last condition, a contradiction.

\begin{table}
\begin{center}
\begin{tabular}{ |c|c|c|c|c|c| }
 \hline
 $H$ & Equation & $t^*$ & $H$ & Equation & $t^*$\\ 
 \hline\hline
 $22$ & $2+x^2+y^2+x y z-z^2=0$ & $0$ & $29$ & $3+x^2+y+x^2 y+y^2+y z^2=0$ & $1.40$ \\ 
 \hline
 $24$ & $4+x^2+y^2+x y z+z^2=0$ & $3.36$ & $29$ & $1+x^2+2 y+x^2 y-y^2+y z^2=0$ & $1$\\ 
 \hline
 $25$ & $1+y^2+x^2 y z+z^2=0$ & $1.55$ & $30$ & $6+y^2+x^2 y z+z^2=0$ & $1.91$ \\ 
 \hline
 $27$ & $3+x^2+x^2 y-y^2+y z^2=0$ & $1.14$ & $30$ & $6+x^2+x^2 y+y^2+y z^2=0$ & $1.86$\\ 
 \hline
 $27$ & $1+x^2+y+x^2 y-y^2+y z^2=0$ & $1$ & $30$ & $6-x^2+x^2 y+y^2+y z^2=0$ & $1.44$\\ 
 \hline
 $28$ & $4+y^2+x^2 y z+z^2=0$ & $1.80$ & $30$ & $2+x^2+y^2+x^2 y z+z^2=0$ & $1.89$\\ 
 \hline
 $28$ & $8+x^2+y^2+x y z+z^2=0$ & $3.61$ & $30$ & $2+x^2+y^2+x^2 y z-z^2=0$ & $1.41$\\ 
 \hline
 $28$ & $2+3 x+x^2 y+y^2+y z^2=0$ & $1.71$ & $30$ & $4+x^2+y+x^2 y-y^2+y z^2=0$ & $1.13$\\ 
 \hline
 $29$ & $5+y^2+x^2 y z+z^2=0$ & $1.86$ & $30$ & $2+x^2+2 y+x^2 y-y^2+y z^2=0$ & $1$\\ 
 \hline
\end{tabular}
\caption{\label{tab:vieta} Equations of size $H\leq 30$ solvable by Algorithm \ref{alg:vieta}.}
\end{center} 
\end{table}

To demonstrate the power of the Vieta jumping method we next prove that 
it can solve all equations in the form
\begin{equation}\label{eq:markwithlin}
x^2+y^2+z^2+ax+by+cz = dxyz + e.
\end{equation}

\begin{proposition}\label{prop:markgen}
There is an algorithm that, given integers $a,b,c,d,e$, determines in finite time (in fact, in time polynomial in $a,b,c,d,e$) whether equation \eqref{eq:markwithlin} has an integer solution.
\end{proposition}
\begin{proof}
We may assume that $d\geq 0$, otherwise replacing $d$ with $-d$ and (say) $a$ with $-a$ leads to an equivalent equation. In fact, we may assume that $d>0$, because the equation is easy for $d=0$. Let us rewrite \eqref{eq:markwithlin} as
\begin{equation}\label{eq:markwithlin2}
(x+a/2)^2+(y+b/2)^2+(z+c/2)^2=dxyz+f,
\end{equation}
where $f=e+(a^2+b^2+c^2)/4$.
Note that for every fixed $x$ equation \eqref{eq:markwithlin} is a $2$-variable quadratic equation, and its solvability can be checked easily. Hence, we can easily check if there exists a solution with $x$ (and similarly with $y$ or $z$) belonging to any bounded range. In particular, we can check whether there is any solution with $|xyz|\leq|f/d|$, and if so, stop, and otherwise assume that $|xyz|>|f/d|$, or $d|xyz|>|f|$. Because $dxyz+f\geq 0$ as a sum of squares, this implies that $xyz\geq 0$.

It suffices to develop an algorithm $A$ that, given $a,b,c,d,e$, determines whether there exist solutions to \eqref{eq:markwithlin} in \emph{non-negative} integers. Indeed, given such an algorithm, we can apply it four times with inputs (i) $a,b,c,d,e$, (ii) $-a,-b,c,d,e$, (iii), $-a,b,-c,d,e$ and (iv) $a,-b,-c,d,e$, and conclude that \eqref{eq:markwithlin} has an integer solution if and only if the algorithm outputs ``Yes'' at least once in these four runs.

We next check for solutions such that $(y+b/2)^2+(z+c/2)^2 \leq f$, or $(x+a/2)^2+(z+c/2)^2 \leq f$, or $(x+a/2)^2+(y+b/2)^2 \leq f$, or $x<5\frac{|a|}{2}$, or $y<5\frac{|b|}{2}$, or $z<5\frac{|c|}{2}$, output ``Yes'' if find a solution, and otherwise assume that the opposite inequalities hold for any solution. In particular, $x\geq 5\frac{|a|}{2}$ implies that
\begin{equation}\label{eq:xbounds}
\frac{6}{5}x \geq x + \frac{a}{2} \geq \frac{4}{5}x.
\end{equation}

Now fix any solution with $x+y+z$ minimal. We may assume that $x\geq y\geq z \geq 0$, because other orderings can be treated similarly. Since we have excluded solutions with $(y+b/2)^2+(z+c/2)^2 \leq f$, equation \eqref{eq:markwithlin2}, considered as a quadratic in $x$, has two positive roots, whose sum is $dyz-a$. If one solution is $x$, another one is $dyz-a-x$. Because we have selected a solution with $x+y+z$ minimal, we must have $x\leq dyz-a-x$, or
$2(x+a/2)\leq dyz$. Together with \eqref{eq:xbounds}, this implies that $2(4x/5)\leq dyz$, or
$$
x \leq \frac{5}{8}dyz.
$$
On the other hand, \eqref{eq:xbounds} implies that $\frac{36}{25}x^2 \geq (x + a/2)^2$. Similarly, $\frac{36}{25}y^2 \geq (y + a/2)^2$, $\frac{36}{25}z^2 \geq (z + a/2)^2$, and, by \eqref{eq:markwithlin2},
\begin{equation}\label{eq:x2y2z2bounds}
\frac{36}{25}(x^2+y^2+z^2)\geq dxyz+f.
\end{equation}
Because $x\geq y\geq z$, this implies that $\frac{36}{25}(3x^2)\geq dxyz+f$, or
$$
x\left(dyz - \frac{108}{25}x\right) \leq -f.
$$
If $dyz - \frac{108}{25}x > 0$, then we must have $-f\geq 0$ and $x\leq -f$. Existence of such solutions can be checked by a direct search, so we may assume that $dyz - \frac{108}{25}x \leq 0$, or
$$
x \geq \frac{25}{108}dyz.
$$ 
Next observe that \eqref{eq:x2y2z2bounds} implies that
$$
x^2+y^2+z^2 - \frac{25}{36}dxyz \geq \frac{25}{36}f,
$$ 
or
$$
\left(x-\frac{25}{72}dyz\right)^2+y^2+z^2-\frac{25^2}{72^2}d^2y^2z^2 \geq \frac{25}{36}f.
$$
Because $\frac{5}{8}dyz \geq x \geq \frac{25}{108}dyz$, we have $\left|x-\frac{25}{72}dyz\right| \leq \max\{\frac{5}{18}dyz,\frac{25}{216}dxyz\} = \frac{5}{18}dyz$, hence
$$
\left(\frac{5}{18}dyz\right)^2+y^2+z^2-\frac{25^2}{72^2}d^2y^2z^2 \geq \frac{25}{36}f,
$$
or
$$
y^2+z^2-\frac{25}{576}d^2y^2z^2  \geq \frac{25}{36}f
$$
Because $y\geq z$, we have $2y^2\geq y^2+z^2$, and we obtain
$$
2y^2 Q(z) \geq \frac{25}{36}f,
$$
where $Q(z)=1-\frac{25d^2}{1152}z^2$. This inequality is possible if either $Q(z) \geq 0$ or $|2y^2 Q(z)|\leq \frac{25}{36}|f|$. Both possibilities can be checked directly in finite time. 
\end{proof}


Up to now, we have applied the Vieta jumping technique for solving equations in at most $3$ variables only. We next demonstrate that the same method works for similar equations in more variables as well.

\begin{proposition}\label{prop:Vieta4var}
Equation
\begin{equation}\label{eq:h34vieta}
x^2+y^2+z^2+t^2 = xyzt - 2
\end{equation}
has no integer solutions.
\end{proposition}
\begin{proof}
This can be proved by solving the optimization problem in Algorithm \ref{alg:vieta} using a computer algebra system, but we will present a human-readable proof in the spirit of the proof of Proposition \ref{prop:markgen}. By contradiction, assume that integer solutions to \eqref{eq:h34vieta} exist, and choose a solution $(x,y,z,t)$ with $|x|+|y|+|z|+|t|$ minimal. We have $xyzt=x^2+y^2+z^2+t^2+2>0$, and, by changing signs, may assume that $x,y,z,t$ are all positive. By symmetry, we may assume that $x\geq y\geq z \geq t$. If $t=1$ or $t=2$, then equation \eqref{eq:h34vieta} has no solutions modulo $9$ hence no integer solutions, so we may assume that $x\geq y\geq z \geq t \geq 3$. Then $4x^2 \geq x^2+y^2+z^2+t^2 = xyzt-2$, hence $4x \geq yzt-\frac{2}{x} \geq yzt-\frac{2}{3}$. Because $4x$ and $yzt$ are integers, this implies that $4x \geq yzt$, or $x \geq \frac{yzt}{4}$. On the other hand, the equation \eqref{eq:h34vieta} has integer solution $(x',y,z,t)$ where $x'=yzt-x$. Because we have selected a solution with $|x|+|y|+|z|+|t|$ minimal, we must have $|x'|\geq x$, or $|yzt-x|\geq x$. Because $yzt>0$, this is possible only if $yzt-x\geq x$, or $x \leq \frac{yzt}{2}$. Thus we have $\frac{yzt}{4} \leq x \leq \frac{yzt}{2}$, or, equivalently, $0 \leq \frac{yzt}{2}-x\leq \frac{yzt}{4}$. This implies that $\left(\frac{yzt}{2}-x\right)^2 \leq \frac{(yzt)^2}{16}$. Hence,
$$
0 = x^2+y^2+z^2+t^2-xyzt+2 = \left(\frac{yzt}{2}-x\right)^2-\frac{(yzt)^2}{4}+y^2+z^2+t^2+2 \leq
$$
$$
\frac{(yzt)^2}{16} -\frac{(yzt)^2}{4} + 3y^2 + 2 = -\frac{3}{16}(yzt)^2 + 3y^2 + 2,
$$
which implies that $\frac{3}{16}(zt)^2 \leq 3+\frac{2}{y^2}\leq 3+\frac{2}{3^2}=\frac{31}{9}$, and $(zt)^2 \leq \frac{31}{9} \cdot \frac{16}{3} < 19$, which is impossible for $z \geq t \geq 3$.
\end{proof}

From now on, we will exclude all equations that 
can be solved by Algorithm \ref{alg:vieta}.

\subsection{$H \geq 25$: Superelliptic equations}\label{sec:superell}

In the previous sections we have covered all equations of size $H\leq 24$. The only equation of size $H=25$ we have not discussed so far is the equation
$$
y^2+y = x^4-3.
$$
This equation reduces to the question whether the determinant of $1+4(x^4-3)=4x^4-11$ can be a perfect square. In this particular case we can easily check that this is impossible\footnote{Indeed, if $|x|\leq 1$ then $4x^4-11<0$, while for $|x|\geq 2$ we have $(2x^2)^2 > 4x^4 - 11 > (2x^2-1)^2$, hence $4x^4-11$ lies between consecutive perfect squares and therefore cannot be a perfect square.}, but 
in general this type of equation can be quite non-trivial. As an example, consider the equation
$$
2y^2+y-x^4-x-2 = 0,
$$
that reduces to the question whether $1^2-4\cdot 2 \cdot (-x^4-x-2)=8x^4 + 8x + 17$ can be a perfect square. 
In 1969, Baker \cite{baker1969bounds} developed a general method for solving equations of the form $y^2=P(x)$.
\begin{theorem}\label{th:ellBaker1}
There is an algorithm for listing all integer solutions (there are finitely many of them) of the equation 
$$
y^2 = P(x) = a_n x^n + \dots + a_1 x + a_0,
$$
provided that all $a_i$ are integers, $a_n\neq 0$, and $P(x)$ has at least three simple (possibly complex) zeros.
\end{theorem}

We next observe that Theorem \ref{th:ellBaker1} can be extended to solve all two-variable Diophantine equations that are quadratic in one of the variables (say $y$). 

\begin{proposition}\label{prod:quady}
There is an algorithm that, given arbitrary polynomials $a(x)$, $b(x)$ and $c(x)$ with integer coefficients, determines all integer solutions to the equation  
\begin{equation}\label{eq:quady}
a(x) y^2+b(x)y+c(x)=0.
\end{equation}
\end{proposition}
\begin{proof}
First assume that $a(x)\equiv 0$. If $b(x)\equiv 0$, then \eqref{eq:quady} reduces to the equation $c(x)=0$ in one variable, which can be solved easily. If $b(x)\not\equiv 0$, then equation $b(x)=0$ has at most finitely many integer solutions, and we can check which of them solve \eqref{eq:quady}. We next assume that $b(x)\neq 0$. In this case,  \eqref{eq:quady} reduces to the question for which $x$ the ratio $\frac{c(x)}{b(x)}$ is  an integer. We may assume that $c(x)$ and $b(x)$ are coprime polynomials over ${\mathbb Q}$, otherwise their common factors can be cancelled. Then there exist polynomials $p(x)$ and $q(x)$ with rational coefficients such that $p(x)c(x)+q(x)b(x)=1$. Let $k>0$ be an integer such that polynomials $kp(x)$ and $kq(x)$ have integer coefficients. If $-y=\frac{c(x)}{b(x)}$ is an integer, then $b(x)(-kp(x)y+kq(x))=k$, hence $b(x)$ is a divisor of $k$. It is left to solve the equations $b(x)=d$ for each (positive or negative) divisors $d$ of $k$, and check which of the resulting $x$ solve \eqref{eq:quady}.

Now assume that $a(x)\not\equiv 0$. Then the equation $a(x)=0$ has at most finitely many integer solutions, and we can check which of them solve \eqref{eq:quady}. We next assume that $a(x)\neq 0$. In this case, equation \eqref{eq:quady} can be solved in $y$ to get
$$
y_{1,2}=\frac{-b(x)\pm \sqrt{b^2(x)-4a(x)c(x)}}{2a(x)},
$$  
and it remains to determine for which $x$ either $y_1$ or $y_2$ is an integer. Let us write $D(x)=b^2(x)-4a(x)c(x)$ as a product $\prod p_i(x)^{a_i}$ of irreducible polynomials $p_i(x)$. Write each $a_i$ as $a_i=2k_i+r_i$, where each $r_i$ is $0$ or $1$. Then $D(x)=Q^2(x)R(x)$, where $Q(x)=\prod p_i(x)^{k_i}$ and $R(x)=\prod p_i(x)^{r_i}$. Because polynomials $p_i(x)$ are irreducible, they have only simple roots and no common roots, hence polynomial $R(x)$ has no repeated roots. Then
\begin{equation}\label{eq:quady_aux}
y_{1,2}=\frac{-b(x)\pm \sqrt{D(x)}}{2a(x)} = \frac{-b(x)\pm Q(x)\sqrt{R(x)}}{2a(x)}.
\end{equation}  
If the degree of $R(x)$ is at least $3$, then it has at least $3$ simple roots, so there are at most finitely many $x$ such that $R(x)$ is a perfect square, and all such $x$ can be listed by an algorithm in Theorem \ref{th:ellBaker1}. We can then check for which of the listed values of $x$ either $y_1$ or $y_2$ is an integer. 

It remains to consider the case when $R(x)$ is at most quadratic. First, assume that polynomial $a(x)$ is non-constant. Then we can write $b(x)=2q_1(x)a(x)+r_1(x)$ and $Q(x)=2q_2(x)a(x)+r_2(x)$ for some polynomials $q_1(x),r_1(x),q_2(x),r_2(x)$ with rational coefficients such that $r_1(x)$ and $r_2(x)$ have degrees strictly smaller than the degree of $a(x)$. Then \eqref{eq:quady_aux} reduces to
$$
y_{1,2}=-q_1(x)\pm q_2(x)\sqrt{R(x)}+\frac{-r_1(x)\pm r_2(x)\sqrt{R(x)}}{2a(x)}.
$$  
Let $k>0$ be an integer such that polynomials $kq_i(x)$ and $kr_i(x)$ have integer coefficients. Then $k y_1$ and/or $ky_2$ can be an integer only if the ratio $r(x)=\frac{-k r_1(x)\pm k r_2(x)\sqrt{R(x)}}{2k a(x)}$ is an integer. But since the degrees of $r_1(x)$ and $r_2(x)$ are strictly smaller than the degree of $a(x)$, and $R(x)$ is at most quadratic, there exists a constant $M$ such that $|r(x)|\leq M$ for all $x$. Then we can solve the equation $r(x)=m$ for all $m$ in the range $[-M,M]$, and check which of the resulting $x$ solve \eqref{eq:quady}.

The remaining case is when $2a(x)=C$ is a non-zero constant. Then consider $|C|$ cases $x=|C|z,$ $\dots,$ $x=|C|z+(|C|-1)$, where $z$ is a new integer variable. In each case $x=|C|z+r$, $y_1$ and/or $y_2$ in \eqref{eq:quady_aux} is an integer if and only if $t=\frac{-b(r)\pm Q(r)\sqrt{R(|C|z+r)}}{C}$ is an integer. But this is equivalent to the equation
$$
(C t +b(r))^2 = Q^2(r)(R(|C|z+r)),
$$ 
which is a quadratic equation in integer variables $t$ and $z$, and can be solved by an algorithm \cite{Alpern}. 
\end{proof} 


Proposition \ref{prod:quady} allows us to exclude all equations of the form \eqref{eq:quady}. This finishes the analysis of the equations of size $H\leq 25$. 

The only $2$-variable equation of size $H=26$ the program returns is
\begin{equation}\label{eq:H26xy}
y^3 = x^4+2.
\end{equation}
This particular equation is easy: substitution $x^2=Y$, $y=X$ reduces it to the equation $Y^2=X^3-2$ of the form \eqref{eq:ellWei}, for which command
$$
sage: EllipticCurve([0,0,0,0,-2]).integral\_points()
$$
returns that the only integer solutions are $X=3$, $Y=\pm 5$, but then $x=\pm \sqrt{Y}$ is not an integer. However, similar equations with an extra $x$ term 
such as, for example, 
$$
y^3 = x^4+x+3,
$$ 
are less trivial. 
%
In the same paper \cite{baker1969bounds} where he proved Theorem \ref{th:ellBaker1}, Baker also proved the existence of an algorithm for listing all integer solutions of the equation
\begin{equation}\label{eq:ym}
y^m = P(x) = a_n x^n + \dots + a_1 x + a_0,
\end{equation}
provided that $m\geq 3$, $n\geq 3$, all $a_i$ are integers, $a_n\neq 0$, and $P(x)$ has at least two simple zeros. In particular, this Theorem covers equation \eqref{eq:H26xy}  because $P(x)=x^4+2$ has $4$ simple zeros. In 1984, Brindza \cite{brindza1984s} showed that the simple zeros condition can be significantly weakened.

\begin{theorem}\label{th:brindza}
Let $m\geq 2$ be an integer, and let $P(x)$ be a polynomial with integer coefficients. Let $\alpha_1, \dots, \alpha_r$ be (possibly complex) roots of $P$ with multiplicities $e_1, \dots, e_r$. Let $m_i = \frac{m}{\text{gcd}(m,e_i)}$, $i=1,\dots,r$, and reorder the roots such that $m_1\geq \dots \geq m_r$. Assume that 
\begin{equation}\label{eq:ym_cond}
(m_1,\dots,m_r) \neq (2,2,1,\dots,1) \quad \text{and} \quad (m_1,\dots,m_r) \neq (t,1,1,\dots,1)
\end{equation}
for any integer $t\geq 1$. Then equation \eqref{eq:ym} has finitely many integer solutions, and there exist an algorithm that lists them all.
\end{theorem}

Condition \eqref{eq:ym_cond} is weaker than Baker's requirement that $P(x)$ has at least two simple zeros. Without this condition, it might be impossible to list all solutions, because there may be infinitely many of them. However, if we only want to decide whether any integer solution exists, Theorem \ref{th:brindza} easily implies the following corollary.

\begin{corollary}\label{cor:brindza}
There is an algorithm that, given any integers $a$ and $m\geq 0$, and any polynomial $P(x)$ with integer coefficients, decides whether equation
\begin{equation}\label{eq:aym}
a y^m = P(x)
\end{equation}
has any integer solution.
\end{corollary}
\begin{proof}
The statement is trivial for $m=0$ and $m=1$, so we may assume that $m\geq 2$. 
Consider $|a|$ cases, $x=|a|z, x=|a|z+1, \dots, x=|a|z+(|a|-1)$, where $z$ is a new integer variable,  and in each case, equation \eqref{eq:aym} is either not solvable because the right-hand side is not divisible by $a$, or, after cancelling $a$, reduces to an equation in the same form \eqref{eq:aym} but with $a=1$. If condition \eqref{eq:ym_cond} holds, then we can list all integer solution by Thereom \ref{th:brindza}. If $(m_1,\dots,m_r)=(2,2,1,\dots,1)$, then $m$ is even and $P(x)$ can be written is the form $P(x)=c Q(x)^m R(x)^{m/2}$, where $c$ is a constant, and $Q(x)$ and $R(x)$ are polynomials with integer coefficients, with $R(x)$ being at most quadratic. Then \eqref{eq:aym} is solvable if and only if there exist $x$ such that $c R(x)^{m/2}$ is a perfect $m$-th power. This is possible only if $c=d^{m/2}$ for some integer $d$, and $dR(x)$ is a perfect square. The last condition is equivalent to the solvability of the quadratic equation $dR(x)=z^2$ in integer variables $x$ and $z$, which can be checked easily \cite{Alpern}. Finally, if $(m_1,\dots,m_r) = (t,1,1,\dots,1)$, then $P(x)=c Q(x)^m L(x)^{m/t}$, where $L(x)=kx+b$ is a linear polynomial, and the question reduces to whether $c (kx+b)^{m/t}$ can be a perfect $m$-th power, which is also easily decidable. 
\end{proof}

Equations in the form \eqref{eq:aym} are known as superelliptic equations. 
Corollary \ref{cor:brindza} allows us to exclude all such equations from further analysis. After this, the program returns no two-variable equations of size $H\leq 30$. 

\subsection{$H \geq 26$: Sum of squares values of a polynomial}\label{sec:sumsquares}

In the previous sections we have solved all equations of size $H \leq 25$ and 
all two-variable equations of size $H\leq 30$. 
The only remaining equations of size $H=26$ are
\begin{equation}\label{eq:H26cubic}
2-yx^2-xyz+y^2+z^2 = 0,
\end{equation}
\begin{equation}\label{eq:H26cubicb}
2+x^2 y+y^2+x^2 z+z^2 = 0
\end{equation} 
and
\begin{equation}\label{eq:h26general}
2+x^3 y+y^2+z^2=0
\end{equation}

Equation \eqref{eq:H26cubic} is similar to ones in Table \ref{tab:vieta} except for the presence of an extra $-yx^2$ term. To get rid of this term, let us replace the variable $z$ with $w-x$ for a new variable $w$. After this linear substitution, equation \eqref{eq:H26cubic} reduces to the equation
\begin{equation}\label{eq:H26transformed}
2 + w^2 - 2 w x + x^2 - w x y + y^2 = 0,
\end{equation}
for which Algorithm \ref{alg:vieta} returns $t^*\approx 5.08 < \infty$. A direct search shows that equation \eqref{eq:H26transformed} has no integer solutions with $\min\{|x|,|y|,|w|\}\leq t^*$, hence the equation \eqref{eq:H26cubic} is not solvable in integers either. This method can be easily automated: try various linear substitutions and apply Algorithm \ref{alg:vieta} to the resulting equations.

Equation \eqref{eq:H26cubicb} can also be simplified by a linear substitution. The fact that $x^2$ is multiplied by $y+z$ suggests a substitution $y \to w-z$ for a new variable $w$. Then \eqref{eq:H26cubicb} reduces to the equation
$$
2 + w^2 + w x^2 - 2 w z + 2 z^2 = 0
$$ 
which is then easily solved by Algorithm \ref{alg:quadtriv} by rewriting it in the form
\begin{equation}\label{eq:H26trans}
w (w+x^2-2 z)=-2-2 z^2
\end{equation}
and doing analysis modulo $m=8$.
In the original variables, \eqref{eq:H26trans} can be written as
$$
(x^2+y-z)(y+z)=-2-2 z^2.
$$
It is straightforward to implement a program that combines Algorithm \ref{alg:quadtriv} with linear substitutions. Table \ref{tab:H30} enumerates all equations up to size $H \leq 30$ that are solvable by this method but not by previous ones. 


\begin{table}
\begin{center}
\begin{tabular}{ |c|c|c|c|c| }
 \hline
 $H$ & Equation & Representation \eqref{eq:quadtriv} & $m$ \\ 
 \hline\hline
 $26$ & $2+x^2 y+y^2+x^2 z+z^2=0$ & $(x^2+y-z)(y+z)=-2-2 z^2$ & $8$ \\ 
 \hline
 $27$ & $3+x^2+x^2 y+y^2-2 z^2=0$ & $(1+y) (-1+x^2+y)=-4+2 z^2$ & $8$ \\ 
 \hline
 $27$ & $3-x^2+x^2 y+y^2+2 z^2=0$ & $(-1+y) (1+x^2+y)=-4-2 z^2$ & $8$ \\ 
 \hline
 $29$ & $5-2 x^2+x^2 y+y^2+z^2=0$ & $(-2+y) (2+x^2+y)=-9-z^2$ & $24$ \\
 \hline
 $29$ & $5+x^3+y^2+x y z+z^2=0$ & $(2+x)(4-2 x+x^2+y z)=3-(y-z)^2$ & $36$ \\
 \hline
 $30$ & $6+x^2+x^2 y+2 y^2+z^2=0$ & $(1+y)(-2+x^2+2 y)=-8-z^2$ & $32$ \\ 
 \hline
 $30$ & $2-x^2-x y+x^3 y+z^2=0$ & $(-1+x) (1+x) (-1+x y)=-1-z^2$ & $4$ \\ 
 \hline
 $30$ & $2+x^2+2 y+x^2 y-2 y^2+z^2=0$ & $(4+x^2-2 y) (1+y)=2-z^2$ & $8$ \\ 
 \hline
 $30$ & $2+2 x+x^3+y^2+x y^2-z^2=0$ & $(1+x) (3-x+x^2+y^2)=1+z^2$ & $4$ \\ 
 \hline
 $30$ & $2-2 x+x^3-y^2+x y^2+z^2=0$ & $(-1+x)(-1+x+x^2+y^2)=-1-z^2$ & $4$ \\ 
 \hline
\end{tabular}
\caption{\label{tab:H30} Equations of size $H\leq 30$ solvable by Algorithm \ref{alg:quadtriv} after linear transformation.}
\end{center} 
\end{table}

Without a computer, it may be quite non-trivial to find a transformation in the second column of Table \ref{tab:H30} that works, and then to do case analysis modulo a large $m$. In particular, the equation 
\begin{equation}\label{eq:h29main}
5+x^3+y^2+x y z+z^2=0
\end{equation}
of size $H=29$ as left open in the first version of this paper \cite{grechuk2021diophantinev1}, and was then solved by Majumdar and Sury in \cite{majumdar2021fruit}.

The last remaining equation of size $H=26$ is \eqref{eq:h26general}. The direct application of Corollary \ref{cor:quadform} seems not to work for this equation. It can be rewritten as $y(x^3+y)=-(z^2+2)$, or, after substitution $y \to -y$, as $y(x^3-y)=z^2+2$. If $z$ is odd, Corollary \ref{cor:quadform} implies that all prime factors $p$ of $z^2+2$ are of the form $p=8k+1$ or $p=8k+3$, which quickly leads to a contradiction. 
However, it is unclear how to use Corollary \ref{cor:quadform} to get a contradiction if $z$ is even. In this case, $z^2+2 = 2 \prod p_i$, where all $p_i$ are at the form $p_i=8k+1$ or $p_i=8k+3$. However, a simple example $x=3$, $y=9$ demonstrates that the  prime factorization of $y(x^3-y)$ may also have this form. 

Proposition \ref{prop:h26sumsquares} below, proved by Will Sawin, solves 
this equation 
by first reducing it to an instance of the following problem: given a polynomial $P$ in one variable, does there exist an integer $x$ such that $P(x)$ can be represented as a sum of two squares? This problem can then be solved using the following Corollary from Proposition \ref{prop:quadform}.

\begin{corollary}\label{cor:evenmult}
Let $a,b,c$ be integers such that $\text{gcd}(a,b,c)=1$, let $D=b^2-4ac$, and let $p$ be an odd prime not dividing $D$ such that $\left(\frac{D}{p}\right)=-1$. Then, for any integers $x$ and $y$, $p$ enters the prime factorization of $m=ax^2+bxy+cy^2$ with even multiplicity.
\end{corollary}
\begin{proof}
Let $k$ be the maximal integer that such $p^k$ is a common divisor of $x$ and $y$, and let $u=\frac{x}{p^k}$ and $v=\frac{y}{p^k}$. Then $p$ is not a common divisor of $u$ and $v$, hence, by Proposition \ref{prop:quadform}, $p$ cannot be a divisor of $au^2+buv+cv^2=\frac{m}{p^{2k}}$. Hence, $p$ enters the prime factorization of $m$ with even multiplicity $2k$.
\end{proof} 

\begin{proposition}\label{prop:h26sumsquares}
Equation \eqref{eq:h26general} has no integer solutions.
\end{proposition}
\begin{proof}

Equation \eqref{eq:h26general} can be written as 
$$
(x^3/2+y)^2+z^2=(x^3/2)^2-2,
$$
or
$$
(x^3+2y)^2 + (2z)^2 = x^6 - 8 = (x^2-2)(x^4 + 2 x^2 + 4).
$$
If $x$ is even then the right-hand side is $8$ times something congruent to $3$ modulo $4$ and cannot be a sum of two squares. If $x$ is odd then $x^2-2$ and $x^4 + 2 x^2 + 4$ are odd and relatively prime, because if they had a common (odd) prime factor $p$, it would also divide $(x^2-2)(x^2+4)-(x^4 + 2 x^2 + 4)=-12$, hence $p$ must be $3$, but $x^2-2$ is never divisible by $3$. But then $x^4 + 2 x^2 + 4$ is odd, positive, and congruent to $3$ mod $4$, hence it must have a prime factor $p$ congruent to $3$ mod $4$ of odd multiplicity. Then $p$ is a prime factor of odd multiplicity of the product $(x^2-2)(x^4 + 2 x^2 + 4)$, hence this product cannot be a sum of two squares by Corollary \ref{cor:evenmult}.
\end{proof}

Proposition \ref{prop:h26sumsquares}, which finishes the analysis of all equations of size $H\leq 26$, 
has been proved using the information about prime factors of $a^2+b^2$. In a similar way, Corollary \ref{cor:evenmult} may be used to determine for which polynomials $P$ there exist $x$ such that $P(x)$ is representable in the forms $a^2\pm 2b^2$, $a^2\pm 3b^2$, and so on. To illustrate the technique, we will solve the equation 
\begin{equation}\label{eq:h29Savin}
1+x^3 y+2 y^2+z^2 = 0
\end{equation}
of size $H=29$.

\begin{proposition}\label{prop:h29Savin}
Equation \eqref{eq:h29Savin} has no integer solutions.
\end{proposition}
\begin{proof}
Let us multiply the equation by $8$ and rewrite as 
$$
(4y+x^3)^2+2(2z)^2=x^6-8=(x^2-2)(x^4+2x^2+4).
$$
Then Corollary \ref{cor:evenmult} implies that if a number of the form $a^2+2b^2$ has a prime factor $p$ congruent to $5$ or $7$ modulo $8$, then the multiplicity of $p$ must be even. If $x$ is even then the right-hand side is $8$ times something congruent to $7$ modulo $8$ and cannot be in the form $a^2+2b^2$. If $x$ is odd then, as established in the proof of Proposition \ref{prop:h26sumsquares}, integers $x^2-2$ and $x^4 + 2 x^2 + 4$ are odd and relatively prime. But then for any $x\neq \pm 1$, $x^2-2$ is odd, positive, and congruent to $7$ mod $8$, hence it must have a prime factor $p$ congruent to $5$ or $7$ mod $8$ of odd multiplicity. Then $p$ is a prime factor of odd multiplicity of the product $(x^2-2)(x^4 + 2 x^2 + 4)$, hence this product cannot be of the form $a^2+2b^2$.
\end{proof}


We next represent the method of proofs of Propositions \ref{prop:h26sumsquares} and \ref{prop:h29Savin} as an (informal) algorithm.

\begin{algorithm}\label{alg:quadform2} 
\begin{itemize}
\item Represent the equation in the form
\begin{equation}\label{eq:quadform2}
QT = aR^2+bRS+cS^2,
\end{equation}
where $a,b,c$ are integers, $R$ and $S$ are arbitrary non-constant polynomials, and $Q,T$ are non-constant polynomials in one variable (say $x$). Denote $U(x_1,\dots,x_n)=aR^2+bRS+cS^2$.
\item Find the smallest integer $K$ (if it exists) such that, for any integer $x$, any common divisor of $Q(x)$ and $T(x)$ must be a divisor of $K$.
\item Then return that the equation is ``Solved'' if it is possible to use Corollary \ref{cor:evenmult} to find integers $m\geq 3$ and $0\leq r_1 < \dots < r_l < m$ such that 
\begin{itemize}
\item[(a)] for any integers $x_1,\dots,x_n$, and any positive divisor $d$ of $U(x_1,\dots,x_n)$
\begin{itemize}
\item[(i)] either $d$ and $U(x_1,\dots,x_n)/d$ have a common divisor that is not a divisor of $K$, or 
\item[(ii)] or $d$ must be equal to some $r_j$ modulo $m$
\end{itemize}
\item[(b)] On the other hand, there is no solution $x_1,\dots,x_n$ of \eqref{eq:quadform2} modulo $m$ such that $|Q|$ and $|T|$ are equal to some $r_j$ modulo $m$.
\end{itemize}
\end{itemize}
\end{algorithm}

We remark that every equation in the form
$$
ay^2+byz+cz^2+P(x)y+Q(x)z+R(x)=0,
$$
where $a,b,c$ are integers, is equivalent to
$$
P^2(4ac-b^2)-4aR(4ac-b^2)+(-bP+2aQ)^2 = (4ac-b^2)(2ay+bz+P)^2+((4ac-b^2)z-bP+2aQ)^2.
$$
The left-hand side is a polynomial in one variable $x$, and, it if is reducible, then this is representation \eqref{eq:quadform2}, and we may try to use Algorithm \ref{alg:quadform2}. The smallest example with $b\neq 0$ for which this method works is the equation
$$
3+3 y+x y+x^2 y-y^2+y z+z^2=0
$$
of size $H=33$, which can be rewritten as
\begin{equation}\label{eq:h33Savin}
-4 (4-x+x^2) (6+3 x+x^2)=(-3-x-x^2-5 z)^2-5 (3+x+x^2-2 y+z)^2,
\end{equation}
and then Algorithm \ref{alg:quadform2} finishes the proof that it has no integer solutions.

Table \ref{tab:H33} enumerates the equations up to size $H \leq 33$ that are left open after application of the previous methods but are solvable by Algorithm \ref{alg:quadform2}. For each equation, we present the corresponding representation \eqref{eq:quadform2}.

\begin{table}
\begin{center}
\begin{tabular}{ |c|c|c| }
 \hline
 $H$ & Equation & Representation \eqref{eq:quadform2}  \\ 
 \hline\hline
 $26$ & $2+x^3 y+y^2+z^2=0$ & $(-2+x^2)(4+2 x^2+x^4)=(x^3 + 2y)^2 + 4z^2$ \\ 
 \hline
 $28$ & $4-x^4+y^2+z^2=0$ & $(-2+x^2)(2+x^2)=y^2+z^2$\\ 
 \hline
 $29$ & $1+x^3 y+2 y^2+z^2=0$ & $(-2+x^2)(4+2 x^2+x^4)=(x^3+4 y)^2+8 z^2$ \\ 
 \hline
 $30$ & $2+x^3 y+y^2+2 z^2=0$ & $(-2+x^2)(4+2 x^2+x^4)=(x^3+2 y)^2+8 z^2$ \\
 \hline
 $33$ & $1+x^3 y+2 y^2+2 z^2 =0$ & $(-2+x^2) (4+2 x^2+x^4)=(x^3+4 y)^2+16 z^2$ \\
 \hline
 $33$ & $1+2 y+x^2 y-2 y^2+2 z-2 z^2=0$ & $(4+2x+x^2)(4-2x+x^2)=(4y-x^2-2)^2+(4z-2)^2$ \\
 \hline
 $33$ & $1+4 x-4 y+x^2 y+y^2+z^2=0$ & $(2-4 x+x^2) (6+4 x+x^2)=(-4+x^2+2 y)^2+4 z^2$ \\
 \hline
 $33$ & $1+4 y+x y+x^2 y-y^2+2 z^2=0$ & $(4+x^2) (5+2 x+x^2)=(4+x+x^2-2 y)^2- 8 z^2$ \\ 
 \hline
 $33$ & $1+4 y+x y+x^2 y-y^2-2 z^2=0$ & $(4+x^2) (5+2 x+x^2)=(4+x+x^2-2 y)^2+8 z^2$ \\ 
 \hline
 $33$ & $3+3 y+x y+x^2 y-y^2+y z+z^2=0$ & \eqref{eq:h33Savin} \\ 
 \hline
 $33$ & $5+x+x^3-y+x^2 y-y^2-z^2=0$ & $(3-2 x+x^2)(7+6 x+x^2)=(-1+x^2-2 y)^2+4 z^2$ \\ 
 \hline
 $33$ & $1+2 x+x^3+y+x y-y^2+z+x z-z^2=0$ & $2(1+2 x)(3+x^2)=(1+x-2 y)^2+(-1-x+2 z)^2$ \\ 
 \hline
\end{tabular}
\caption{\label{tab:H33} Equations of size $H\leq 33$ solvable by Algorithm \ref{alg:quadform2}.}
\end{center} 
\end{table}


If an equation reduces to the question whether a polynomial $P(x)$ can take values representable as a sum of squares (or other quadratic form), but 
Algorithm \ref{alg:quadform2} is not applicable,
this is an indication that an equation may be solvable. To illustrate this, consider equation
\begin{equation}\label{eq:h27hassol}
1+x+x^3+x y+y^2+2 z^2=0
\end{equation}
of size $H=27$. 
It is equivalent to $(2y+x)^2+2(2z)^2=-4x^3+x^2-4x-4$, hence we need to know if there exist $x$ such that $P(x)=-4x^3+x^2-4x-4$ is representable as $a^2+2b^2$, with some parity constraints on $a$ and $b$. 
Polynomial $P(x)$ is irreducible, and divisibility analysis modulo $8$ does not preclude $P(x)$ from having only prime factors allowed by Corollary \ref{cor:quadform}. Motivated by this, we performed a computer search and found that $x=-30, y=-51, z=107$ is a solution. The program did not discover this automatically, because it searched only for solutions up to $100$. 


Motivated by this, we implemented the following algorithm for looking for larger integer solutions.

\begin{algorithm}\label{alg:largesol} 
Choose some positive parameters $N_1$ and $N_2$. Then, given a polynomial Diophantine equation $P(x_1,\dots,x_n)=0$, do the following:
\begin{itemize}
\item[(i)] Check whether the equation has any integer solution with $\max\{|x_1|,\dots,|x_n|\}\leq N_1$. If any integer solution is found, return the solution and stop.
\item[(ii)] Call a variable $x_i$ ``special'' if substitution a constant instead of $x_i$ makes the equation (at most) quadratic. Then for each special variable $x_i$ and for each constant $c$ with $|c|\leq N_2$ check whether the equation has a solution with $x_i=c$ using known algorithms for quadratic equations \cite{Alpern, grunewald1981solve}. If any solution found, return the solution and stop.
\end{itemize} 
\end{algorithm}

\begin{table}
\begin{center}
\begin{tabular}{ |c|c|c| }
 \hline
 $H$ & Equation & Solution  \\ 
 \hline\hline
 $27$ & $1+x+x^3+x y+y^2+2 z^2=0$ & $x=-30$, $y=-51$, $z=107$ \\ 
 \hline
 $28$ & $4+x^3 y+y^2+z^2=0$ & $x=-9$, $y=17$, $z=110$ \\ 
 \hline
  $29$ & $7+x+x^3+y^2 z+z^2=0$ & $x=-224$, $y=-16$, $z=-3483$ \\
 \hline
 $29$ & $7-x+x^3+2 z+y^2 z=0$ & $x=61$, $y=39$, $z=-149$ \\ 
 \hline
 $30$ & $6+x^3-y^2+x y^2-z^2=0$ & $x=20$, $y=25$, $z=141$ \\ 
 \hline
 $30$ & $4+x^3-x y+y^3+z+z^2=0$ & $x=-64$, $y=-30$, $z=539$ \\ 
 \hline
 $30$ & $2+x^2+x^3+y^3-y z^2=0$ & $x=10256$, $y=23866$, $z=24795$ \\
 \hline
 $31$ & $3+y^2+x^2 y z-2 z^2=0$ & $x=5$, $y=163$, $z=2044$ \\ 
 \hline
 $31$ & $3+4 x+2 y+x^2 y+x z^2 =0$ & $x=21$, $y=-114$, $z=49$ \\ 
 \hline
 $31$ & $1+x^2-x^2 y^2+z+2 z^2 =0$ & $x=113$, $y=45$, $z=-3595$ \\ 
 \hline
 $31$ & $1+2 x+x^3+5 z+y^2 z =0$ & $x=-23$, $y=9$, $z=142$ \\ 
 \hline
 $31$ & $3+x^3+y^2-x y^2+x z+z^2 =0$ & $x=5$, $y=61$, $z=119$ \\ 
 \hline
 $31$ & $1+y+x^2 y-y^2+x y z-2 z^2 =0$ & $x=3$, $y=161$, $z=170$ \\ 
 \hline
 $31$ & $1+2 x-x^2-y+x^2 y-y^2+x z^2 =0$ & $x=131$, $y=-251$, $z=183$ \\ 
 \hline
 $31$ & $1-x^2+x^3-y+x y+y^2+y z^2 = 0$ & $x=-39$, $y=-277$, $z=-30$ \\ 
 \hline
 $31$ & $3+x+x^3+y-x y+y^3+z^2 = 0$ & $x=42$, $y=-98$, $z=929$ \\ 
 \hline
 $31$ & $1+x^2+x^3+y+y^3+y z+z^2 = 0$ & $x=-31$, $y=-38$, $z=-271$ \\ 
 \hline
 $32$ & $2 + 5 y + x^3 y - z^2=0$ & $x=9$, $y=113$, $z=288$ \\ 
 \hline
 $32$ & $2+x^4+y+y^2+y z^2=0$ & $x=280$, $y=-77254$, $z=396$ \\ 
 \hline
 $32$ & $2-x+x^3+x y^2+x y z-z^2=0$ & $x=7858$, $y=934$, $z=-66444$ \\ 
 \hline
\end{tabular}
\caption{\label{tab:H32sol} Equations of size $H\leq 32$ with solution found by Algorithm \ref{alg:largesol} with $N_1=10^5$ and $N_2=100$.}
\end{center} 
\end{table}

Table \ref{tab:H32sol} lists equations of size $H\leq 32$ for which we have found a solution using Algorithm \ref{alg:largesol} with parameters $N_1=10^5$ and $N_2=100$. 
In fast, all the solutions listed in Table \ref{tab:H32sol} are found by part (i) of Algorithm \ref{alg:largesol}. The smallest equations where we have used part (ii) are the equations 
$$
1 + x^2 + y^2 + x^2 y z + 2 z^2,
$$
$$
3 + y + y^2 + x^2 y z + 2 z^2,
$$
and
$$
1 + y + y^2 - z + x^2 y z + 2 z^2
$$
of size $H=33$,
for which Algorithm \ref{alg:largesol} (ii) returns solutions $(x,y,z)=$
$$
(7,157088322340,-3208559045),
$$
$$
(9,-1441281757325313892736191569076105,58354112142371754255532446172358469)
$$
and
$$
(9,-521215429444971521761892888274049916,21102788171418790035516786293649655539),
$$
respectively.
We did not check whether the solutions we found are the smallest ones. Table \ref{tab:H32sol} finishes the analysis of all equations of size $H \leq 28$.

\subsection{$H \geq 29$: More advanced versions of the above algorithms}\label{sec:adv}

In the previous sections we have solved all equations of size $H\leq 28$, so we next discuss the remaining equations of size $H=29$. We start with the equation
\begin{equation}\label{ex:h29jump}
5+x^2+y^2+x y z-2z^2 = 0.
\end{equation}
This equation looks similar to the ones solvable by the Vieta jumping technique, but Algorithm \ref{alg:vieta} returns $t^*=\infty$ and therefore is not applicable. Instead, we rewrite this equation in the form \eqref{eq:prodform} as
\begin{equation}\label{ex:h29repr}
(8+y^2) (-2+z) (2+z)=-12+(2 x+y z)^2 
\end{equation}
and then run Algorithm \ref{alg:quadform}. It indeed works, returning a contradiction modulo $m=144$. But how to find representations like \eqref{ex:h29repr}? The following Algorithm gives a method that works in many cases.

\begin{algorithm}\label{alg:quad2}~
The input is a polynomial Diophantine equation in variables $x_1,\dots,x_n$. For each $1\leq i \leq n$, check whether the equation is quadratic in $x_i$ with a constant coefficient near the quadratic term: that is, whether it can be written in the form
$$
a x_i^2 + P x_i + Q = 0
$$
where $a$ in an integer and $P,Q$ are polynomials in other variables. Choose a parameter $B>0$, say $B=100$. Then, for each constant $C$ with $|C|\leq B$, 
multiply the equation by $4a$, rewrite as
\begin{equation}\label{eq:genquadrepr}
P^2 - 4 a Q + C = (2 a x_i + P)^2 + C,
\end{equation}
and check whether the polynomial $P^2 - 4 a Q + C$ is reducible. If yes, 
this is representation \eqref{eq:prodform}, so we can run Algorithm \ref{alg:quadform}. If it works for any pair $(i,C)$, return that the equation is solved. 
\end{algorithm}

Table \ref{tab:H33quad} lists the equations up to size $H \leq 33$ that we have solved by Algorithm \ref{alg:quad2}. For each equation, we list the representation \eqref{eq:genquadrepr} that works, and the modulus $m$ for which Algorithm \ref{alg:quadform} returns a contradiction. 

\begin{table}
\begin{center}
\begin{tabular}{ |c|c|c|c| }
 \hline
 $H$ & Equation & Representation \eqref{eq:genquadrepr} & $m$ \\ 
 \hline\hline
 $29$ & $5+x^2+y^2+x y z-2 z^2=0$ & $(8+y^2) (-2+z) (2+z)=-12+(2 x+y z)^2$ & $144$ \\ 
 \hline
 $29$ & $1+2 y+x^2 y+2 x z+2 z^2 =0$ & $-(2+x^2) (-1+2 y)=4+(x+2z)^2$ & $4$ \\ 
 \hline
 $29$ & $1+x^2+x^2 y+y^2+z^2-y z^2=0$ & $(-2+(x-z)^2) (-2+(x+z)^2)=8+(x^2+2 y-z^2)^2$ & $16$ \\
 \hline
 $31$ & $1+3 y+x^2 y+2 x z-2 z^2=0$ & $(3+x^2)(1+2 y)=1+(x-2z)^2$ & $4$ \\
 \hline
 $31$ & $1+x^2+3 y+2 x^2 y+z^2 =0$ & $-(3+2 x^2) (1+2 y)=-1+2 z^2$ & $8$ \\
 \hline
 $31$ & $3+x^2+x^2 y+y^2+z^2-y z^2=0$ & $(-2+(x-z)^2) (-2+(x+z)^2)=16+(x^2+2 y-z^2)^2$ & $16$ \\
 \hline
 $32$ & $2+x^2+x y+x^3 y+z+z^2=0$ & $-(1+x^2) (1+x y)=1+z+z^2$ & $3$ \\
 \hline
 $32$ & $4+x^2+y^2+x y z-3 z^2=0$ & $(-2+x)(2+x)(12+z^2)=-32+(2 y-x z)^2$ & $64$ \\
 \hline
 $33$ & $1+2 y^2+x^2 y z+2 z^2=0$ & $(-2+x) (2+x) (4+x^2) z^2=8+(4 y+x^2 z)^2$ & $16$ \\
 \hline
 $33$ & $5+2 y+x^2 y+2 x z-2 z^2=0$ & $(2+x^2)(1+2 y)=-8+(x-2z)^2$ & $8$ \\
 \hline
 $33$ & $1-4 y+x^2 y+2 x z+2 z^2=0$ & $(-2+x)(2+x)(-1+2 y)=2-(x+2 z)^2$ & $8$ \\
 \hline
 $33$ & $1+4 y+x^2 y+2 x z-2 z^2=0$ & $(4+x^2)(1+2 y)=2+(x-2z)^2$ & $8$ \\
 \hline
 $33$ & $5+x^2 y+y^2+2 x z-y z^2=0$ & $(-2+(x-z)^2) (2+(x+z)^2)=16+(x^2+2 y-z^2)^2$ & $32$ \\
 \hline
 $33$ & $1+x^2 y+2 y^2+2 x z-y z^2=0$ & $(-4+(x-z)^2)(4+(x+z)^2)=-8+(x^2+4 y-z^2)^2$ & $16$ \\
 \hline
 $33$ & $1+x y+x^3 y+x^2 z+z^2=0$ & $(1+x^2) (-1+x^2-4 x y)=3+(x^2+2 z)^2$ & $24$ \\
 \hline
 $33$ & $1+x^2+x^2 y-2 y^2+z^2-y z^2=0$ & $(4+(x-z)^2) (4+(x+z)^2)=8+(x^2-4 y-z^2)^2$ & $16$ \\
 \hline
 $33$ & $1+x^2+2 y+x^2 y+2 y^2-y z^2=0$ & $(-2+(x-z)^2) (-2+(x+z)^2)=8+(2+x^2+4 y-z^2)^2$ & $16$ \\
 \hline
\end{tabular}
\caption{\label{tab:H33quad} Equations of size $H\leq 33$ solvable by Algorithm \ref{alg:quad2}.}
\end{center} 
\end{table}

The next example we consider is the equation
\begin{equation}\label{eq:h29a}
1+2 x+x y+x^2 y+2 z+y^2 z = 0.
\end{equation}
This equation is linear in $z$, can be rewritten as 
$$
-z(y^2+2)=1+2 x+x y+x^2 y,
$$
and reduces to the question whether $y^2+2$ can be a divisor of $1+2 x+x y+x^2 y$. The idea is that if $d$ is any common divisor of $y^2+2$ and $1+2 x+x y+x^2 y$, then, modulo $d$, 
$$
0 \equiv (x^2+x)y+2x+1 \equiv (x^2+x)y^2+(2x+1)y \equiv (x^2+x)(-2)+(2x+1)y 
$$
$$
\equiv 4(x^2+x)-2(2x+1)y + (y^2+2) = (2x-y+1)^2+1.
$$
The rest is easy. Because $x^2+x$ is even, it follows from \eqref{eq:h29a} that $y^2+2$ is odd. Because $y^2+2$ is a divisor of a sum of squares, all its prime factors are $1$ modulo $4$, hence $y^2+2 \equiv 1(\text{mod }4)$, a contradiction. 

This argument can be automated by the following more advanced version of Algorithm \ref{alg:quadform}, which analyses \emph{common} divisors of $P_j$ and $Q$ instead of ``just'' divisors of $Q$.

\begin{algorithm}\label{alg:quadform3}~ 
Assume that the equation is presented in the form
\begin{equation}\label{eq:prodform3}
\prod_{j=1}^k P_j = Q,
\end{equation}
where $P_1,\dots,P_k,Q$ are non-constant polynomials in variables $x_1,\dots,x_n$ with integer coefficients. Choose a bound $B>3$. For integers $m=3,4,\dots, B$, do the following: 
\begin{itemize}
\item[(a)] For each $j=1,2,\dots,k$ form a set $R_j=\{r_1, \dots ,r_{l_j}\}$ of all possible residues that positive \emph{common} divisors of $Q$ and $P_j$ can have modulo $m$;
\item[(b)] if there is no solution $x_1,\dots,x_n$ of \eqref{eq:prodform3} modulo $m$ such that all residues of $|P_j(x_1,\dots,x_n)|$ modulo $m$ belong to $R_j$, conclude that the equation has no integer solutions.
\end{itemize}
\end{algorithm}

Now, any equation linear in some variable $x_i$ can be represented in the form $x_i \cdot P = Q$ where $P,Q$ are polynomials in other variables. This is representation \eqref{eq:prodform3}, and we can run Algorithm \ref{alg:quadform3}. The results up to $H \leq 33$ are presented in Table \ref{tab:H33lin}.

\begin{table}
\begin{center}
\begin{tabular}{ |c|c|c|c| }
 \hline
 $H$ & Equation & $m$ \\ 
 \hline\hline
 $29$ & $1+2 x+x y+x^2 y+2 z+y^2 z =0$ & $4$ \\ 
 \hline
 $33$ & $1+4 y+x^2 y+2 x z^2=0$ & $8$ \\ 
 \hline
 $33$ & $1+2 x+6 y+x^2 y-2 z^2=0$ & $4$ \\ 
 \hline
 $33$ & $1+2 y-x^2 y^2+2 z+x^2 z=0$ & $4$ \\ 
 \hline
 $33$ & $5-2 x+x y+x^2 y+2 z+y^2 z=0$ & $4$ \\ 
 \hline
 $33$ & $1+2 x+x y+x^2 y-4 z+y^2 z=0$ & $8$ \\ 
 \hline
 $33$ & $1-2 x+x y+x^2 y+4 z+y^2 z=0$ & $8$ \\ 
 \hline
 $33$ & $1+x-x^2+x y+x^2 y+3 z+y^2 z=0$ & $4$ \\ 
 \hline
\end{tabular}
\caption{\label{tab:H33lin} Equations of size $H\leq 33$ solvable by Algorithm \ref{alg:quadform3}.}
\end{center} 
\end{table}


The listed algorithms automatically solve all equations up to $H\leq 30$ except for two equations of size $H=29$. The first one is the equation
$$
3+x^2+y+x^2 y+y^2-y z^2 = 0
$$
of size $H=29$. It is convenient to do the substitution $y\to -y$ and rewrite this equation in the form
\begin{equation}\label{eq:h29c}
y(x^2-z^2-y+1)=x^2+3.
\end{equation}
Corollary \ref{cor:quadform} states that the prime factors of $x^2+3$ other than $2$ and $3$ must be in the form $p=3k+1$, but the example $x=y=7$, $z=6$ shows that the left-hand side can also have only such prime factors. However, \eqref{eq:h29c} still has no integer solutions, as we prove next.

\begin{proposition}\label{prob:h29c}
Equation \eqref{eq:h29c} has no integer solutions.
\end{proposition}
\begin{proof}
Denote $x^2-z^2-y+1$ by a new variable $t$. Then $ty=x^2+3$ by \eqref{eq:h29c}. Now, by definition of $t$, $t=x^2+3-3-z^2-y+1=ty-z^2-y-2$, which can be written as $(t-1)(y-1)=z^2+3$. Hence, equation \eqref{eq:h29c} reduces to a nice system of equations
$$
\begin{cases} 
(t-1)(y-1)=z^2+3 \\
ty=x^2+3
\end{cases}
$$ 
Now we can use Corollary \ref{cor:quadform} to prove that this system has no integer solutions. It is clear that $y$ and $t$ have the same sign, and we may assume that they are positive, otherwise substitution $y'=1-y$, $t'=1-t$ leads to the same system with positive variables. First assume that both $y$ and $t$ are even. Because $x^2+3$ is not divisible by $8$, this is possible only if both $y$ and $t$ are equal to $2$ modulo $4$. But then both $t-1$ and $y-1$ are equal to $1$ modulo $4$, and so is their product $z^2+3$. But then $z^2$ is $2$ modulo $4$, a contradiction. Similarly, if both $y$ and $t$ are odd, then both $y-1$ and $t-1$ are even, which is possible only if they both are equal to $2$ modulo $4$. But then $y$ and $t$ are both $3$ modulo $4$, and their product $x^2+3$ is $1$ modulo $4$, a contradiction.

Finally, assume that $y$ is even and $t$ is odd (the case when $y$ is odd and $t$ is even is similar). Then $x^2+3$ is even, and must therefore be $4$ modulo $8$, which implies that $y=4a$ for some odd integer $a$. For the same reason, $t-1=4b$ for an odd integer $b$, and we have 
$$
\begin{cases} 
4b(4a-1)=z^2+3 \\
(4b+1)4a=x^2+3.
\end{cases}
$$ 
By Corollary \ref{cor:quadform}, all prime factors of odd positive integers $a$, $b$, $4a-1$ and $4b+1$ must be either $3$ or in the form $p=3m+1$. This implies that none of these numbers can be equal to $2$ modulo $3$. But this is possible only if $a$ is $1$ modulo $3$ and $b$ is $0$ modulo $3$. But in this case both $4b$ and $4a-1$ are divisible by $3$, hence their product $z^2+3$ is divisible by $9$, which is a contradiction.  
\end{proof}

The last equation of size $H\leq 30$ that the listed algorithm could not solve is the equation
\begin{equation}\label{eq:h29d}
1+x+x^3+x y^2-z+z^3 =0
\end{equation}
of size $H=29$. This equation turned out to be solvable in integers. We performed a computer search up to $|z|\leq 10^6$, found no solutions, and listed this equation as open in the first version of this paper \cite{grechuk2021diophantinev1}. But then Andrew R. Booker performed a check up to $|z|\leq 10^7$ and found a solution 
\begin{equation}\label{eq:h29dsol}
x=-4280795, \quad y=4360815, \quad z=5427173. 
\end{equation}
This finishes the analysis of the equations of size $H \leq 30$.

\subsection{$H \geq 31$: Entering the open territory}\label{sec:open}

In the previous section we finished the analysis of equations of size $H \leq 30$. For $H=31$, 
we have the first two-variable equations not covered by Proposition
\ref{prod:quady} 
and Corollary \ref{cor:brindza}. They are
\begin{equation}\label{eq:H31b}
y^3+x^3y-2x-3=0,
\end{equation}
\begin{equation}\label{eq:H31c}
y^3+x^3y+y+2x-1=0,
\end{equation}
\begin{equation}\label{eq:H31d}
xy^3-2y+x^3+x+1=0
\end{equation}
and
\begin{equation}\label{eq:H31a}
y^3+xy-x^4-3=0.
\end{equation}
However, equations \eqref{eq:H31b}-\eqref{eq:H31d} are covered by another deep result, see e.g. \cite{walsh1992quantitative} for a proof.

\begin{theorem}\label{th:Runge}
Let ${\cal F}$ be a family of polynomials
$$
P(x,y) = \sum_{i=0}^m \sum_{j=0}^n a_{ij} x^i y^j
$$
with integer coefficients $a_{ij}$ of degree $m>0$ in $x$ and $n>0$ in $y$ which are irreducible\footnote{That is, $P(x,y)$ is not a product of non-constant polynomials with rational coefficients} over ${\mathbb Q}[x,y]$, and satisfy at least one of the following conditions:
\begin{itemize}
\item[(C1)] either there exists a coefficient $a_{ij}\neq 0$ of $P$ such that $ni+mj>mn$, or 
\item[(C2)] the sum of all monomials $a_{ij}x^iy^j$ of $P$ for which $ni + mj = nm$ can be decomposed into a product of two non-constant relatively prime polynomials in ${\mathbb Z}[x,y]$.
\end{itemize}
Then there is an algorithm that, given any polynomial $P\in {\cal F}$, determines all integer solutions of equation $P=0$.
\end{theorem}

In 1887, Runge \cite{runge1887ueber} proved that for every $P\in {\cal F}$ equation $P=0$ has at most finitely many integer solutions. Theorem \ref{th:Runge} is an effective version of this result. In particular, it implies that Problem \ref{prob:main} is solvable for this class of equations. 

Now, equations \eqref{eq:H31b} and \eqref{eq:H31c} have degrees $m=n=3$ in $x$ and $y$, and coefficient $a_{31}=1\neq 0$. But then $ni+mj=3\cdot 3 + 3 \cdot 1 > 9 = mn$, (C1) holds, and Theorem \ref{th:Runge} is applicable. Equation \eqref{eq:H31d} has $a_{13}\neq 0$, and Theorem \ref{th:Runge} is applicable by a similar argument.  
%
In contrast, equation \eqref{eq:H31a} is not covered by Theorem \ref{th:Runge}. 
Indeed, for this equation $m=4$, $n=3$, and the only non-zero coefficients are $a_{03}, a_{11}, a_{40},$ and $a_{00}$. This implies that (C1) does not hold. In (C2), the monomials with $ni + mj = nm$ are $y^3$ and $-x^4$, and their sum $y^3-x^4$ is irreducible\footnote{This can be checked directly or using the following easy sufficient condition of Ehrenfeucht \cite{ehrenfeucht1958kryterium}: if degrees of polynomials $f(x)$ and $g(y)$ are relatively prime, then $f(x)-g(y)$ is irreducible.} and therefore is not a product of relatively prime polynomials. 
However, equation \eqref{eq:H31a} is covered by Algorithm \ref{alg:quadtriv}. This finishes the analysis of all $2$-variable equations of size $H\leq 31$.

All the listed methods can be combined into one unified algorithm.

\begin{algorithm}\label{alg:unified}~ 
Choose some positive parameters $N_0,N_1,N_2,A,M,B$ and $S$. Then, given a polynomial Diophantine equation $P(x_1,\dots,x_n)=0$, do the following: 
\begin{itemize}
\item[(i)] Check whether is has a solution with $\max_i |x_i| \leq N_0$. If yes, stop and report that equation is solvable.
\item[(ii)] Check whether the equation belongs to trivial families listed in section \ref{sec:trivial}, that is, whether it is of the form \eqref{eq:lingen}, of the form \eqref{eq:polprod} for $k\geq 2$ and some $|a|\leq A$, has empty or bounded set of real solutions, or is not solvable modulo $m$ for some $2 \leq m \leq M$. If yes, stop.
\item[(iii)] Check whether it belongs to some family for which an algorithm is well-known, that is, whether it is quadratic and covered by Theorem \ref{th:quadratic}, or $2$-variable one covered by Theorem \ref{th:notabsirr}, \ref{th:cubic2var}, \ref{th:ellBaker1}, \ref{th:brindza} or \ref{th:Runge}. If yes, stop.
\item[(iv)] Run Algorithms \ref{alg:quadtriv}, \ref{alg:vieta}, \ref{alg:quadform2}, \ref{alg:largesol} (with parameters $N_1$ and $N_2$), \ref{alg:quad2} (with parameter $B$) and \ref{alg:quadform3} for this equation. If any of these algorithms works, stop and report the answer.
\item[(v)] For each integer $s \neq 0$ with $|s|\leq S$, transform the equation by linear substitutions (a) $x_i \to x_i + s$ for each variable $i$, and (b) $x_i \to x_i + s x_j$ for each pair of variables $i,j$. Do one substitution at a time. Then repeat part (iv) for each of the resulting equations.
\item[(vi)] Give up and report that the program cannot determine whether the equation is solvable in integers.
\end{itemize}
\end{algorithm}

Algorithm \ref{alg:unified} with parameters $N_1=10^5$, $M=150$, $N_0=N_2=A=B=100$ and $S=2$ automatically solves all equations of size $H \leq 31$, except for equation \eqref{eq:h29d} which has a larger solution \eqref{eq:h29dsol}, equation \eqref{eq:h29c} that we have solved in Proposition \ref{prob:h29c}, and three equations of size $H=31$ that we will discuss next. We start from the equation
\begin{equation}\label{eq:h31e}
3-x^2+x^2 y+y^2+y z-y z^2 = 0.
\end{equation}
One may note that this equation is equivalent to $y(x^2+y+z-z^2)=x^2-3$ and use Corollary \ref{cor:quadform} to analyse the prime factors of $x^2-3$. However, it turns out that 
this analysis does not lead to a contradiction. Instead, one may note that the cubic terms of this equation are $yx^2-yz^2=y(x-z)(x+z)$, which with new variables $x-z=u$ and $x+z=v$ reduces to $yuv$. All the other terms are at most quadratic, which is an indication that the Vieta jumping technique may work. And it indeed works, but requires some effort.

\begin{proposition}\label{prob:h31a} 
Equation \eqref{eq:h31e} has no integer solutions.
\end{proposition}
\begin{proof}
After substitution $z+x\to u$, $z-x \to v$ and multiplication by $4$, the equation reduces to $4y^2-u^2-v^2+2uv+2yv+2yu-4yuv+12=0$. Then we do substitution $u \to u/2, v \to v/2, y \to y/2$ and multiply by $4$ to get a new equation $48 - u^2 + 2 u v - v^2 + 2 u y + 2 v y - 2 u v y + 4 y^2 = 0$, and we need to prove that it has no solutions in even integers. Next we do the substitution $u \to u + 1, v \to v + 1, y \to y + 1$ to get the equation $54 - u^2 - v^2 + 10 y - 2 u v y + 4 y^2=0$ that should have no solutions in odd integers. Finally, we do the substitution $y \to y/2$ to get the equation 
$$
54 - u^2 - v^2 + 5 y - u v y + y^2 = 0,
$$ 
and we will prove that this equation has no solutions such that $u$ and $v$ are odd while $y$ is even. 
Let us call such solutions ``eligible''. Let us choose an eligible solution with the smallest $|u|+|v|+|y|$. From symmetry, we may assume $u\geq v > 0$. First consider the case $y>0$. Then equation, as a quadratic in $y$, has another integer solution $y'=uv-5-y$. Note that if $u$ and $v$ are odd and $y$ even, then $y'$ is even, hence the new solution is eligible. Thus we must have $|y'|\geq y$. Note also that $yy' = 54-u^2-v^2$. The case of small $u$ and $v$ can be checked by an easy computer search, so we may assume that $54-u^2-v^2<0$. Then $y'<0$, and the inequality $|y'|\geq |y|$ reduces to $-y'\geq y$, or $-(uv-5-y)\geq y$, or $uv\leq 5$. However, a direct search returns no eligible solution with $u\geq v > 0$ but $uv\leq 5$, a contradiction. Now consider the case $y<0$. Then the equation as quadratic in $u$ has another integer solution $u'=-vy-u$. If $y$ is even and $v$ and $u$ are odd, then $u'$ is odd, hence the new solution is eligible. Hence we must have  
$|u'|\geq u$. Because $u'+u=-vy>0$, this implies that $u'>0$. Then $u'u\geq u^2$, or $v^2-y^2-5y-54 \geq u^2 \geq v^2$, or $-y^2-5y-54 \geq 0$, a contradiction.
\end{proof}

The next equation left open by Algorithm \ref{alg:unified} is the equation
\begin{equation}\label{eq:h31f}
3+x^2+x^2 y+2 y^2-y z^2 = 0
\end{equation}
which can be solved by a similar method.

\begin{proposition}\label{prob:h31f} 
Equation \eqref{eq:h31f} has no integer solutions.
\end{proposition}
\begin{proof}
Substitutions $y=t/2$, $x=(v-u)/2$ and $z=(v+u)/2$ and multiplication by $4$ reduce the equation to $12 + (u-v)^2 - 2 u v t + 2t^2 = 0$. Let us prove that, more generally, that each of two equations 
\begin{equation}\label{eq:pair}
12+(u\pm v)^2-2uvt+2t^2=0
\end{equation} 
has no integer solutions. By contradiction, assume that solutions exists, and let us choose one with $|u|+|v|+|t|$ minimal. By changing signs and by symmetry we may assume that $u\geq v \geq 0$, and then clearly $t\geq 0$. First assume that $2t^2 \geq u^2$. Then $2uvt = 2t^2 + (u\pm v)^2 +12 \leq 2t^2 + 4u^2 +12 \leq 10t^2 + 12 \leq 12t^2$ if $t>2$ (cases $t=0,1,2$ can be checked separately). Thus $uv/6 \leq t$. On the other hand, the equation, considered as quadratic in $t$, has another solution $t'=uv-t=\frac{12+(u\pm v)^2}{2t}>0$. Because $|u|+|v|+|t|$ was minimal, we have $t'\geq t$, or $uv-t \geq t$, or $uv \geq 2t$. Now multiply the equations \eqref{eq:pair} by $2$ and rewrite as $(uv-2t)^2-(uv)^2+2(u\pm v)^2+24=0$. We have $uv-2t\geq 0$ but $uv-2t \leq uv-2(uv/6)=2(uv)/3$. Hence $(2uv/3)^2-(uv)^2+2(u\pm v)^2+24\geq 0$, but this is possible only for small values of $u,v$ that can be checked directly. This finishes the case $2t^2 \geq u^2$. Note that we did this by considering the equation as quadratic in $t$. In exactly the same way, the case $u^2 \geq 2t^2$ can be treated by considering the equation as quadratic in $u$. 
\end{proof}

The last equation of size $H=31$ left open by Algorithm \ref{alg:unified} is the equation that the program outputs as
$$
3+x^3 y+y^2+z^3 = 0.
$$
With substitution $y \to -y$, we can rewrite this equation in an even nicer form
\begin{equation}\label{eq:h31main}
y(x^3-y)=z^3+3.
\end{equation}
This equation is cubic in both $x$ and $z$, and it is not clear how to apply the methods discussed above (like the information about prime factors of quadratic forms or the Vieta jumping technique) to prove that it has no integer solutions. 
On the other hand, we have performed a computer search and found no solutions to \eqref{eq:h31main} with $|z|\leq 10^8$.
We leave the solvability of this equation to the reader as an open question. 

\begin{question}\label{qu:smallest}
Do there exist integers $x,y,z$ satisfying \eqref{eq:h31main}? 
\end{question}

Equation \eqref{eq:h31main} can be rewritten as
$$
(y-x^3/2)^2-x^6/4+z^3+3=0.
$$
We next do some heuristic analysis of how many solutions to expect to a more general equation
\begin{equation}\label{eq:powersum}
a_1 x_1^{d_1} + a_2 x_2^{d_2}+\dots + a_n x_n^{d_n}=C,
\end{equation}
provided that there are no obstructions (such as divisibility, quadratic residues, Vieta jumping, or any other obstructions) to the solution’s existence. Let us choose some large constant $B$ and look at the solutions such that $B/2 < \max_i |a_i x_i^{d_i}| \leq B$. There are $O(B^{1/d_i})$ choices of each variable $x_i$ such that $|a_i x_i^{d_i}|\leq B$, resulting in about $K B^{S}$ combinations of variables, where 
\begin{equation}\label{eq:Sdef}
S=\sum_{i=1}^n\frac{1}{d_i},
\end{equation}
and $K>0$ is a constant depending on the equation. From this, we need to subtract about $K (B/2)^{S}$ combinations with $\max_i |a_i x_i^{d_i}| \leq B/2$, but the result is still $K_1 B^{S}$ for a different constant $K_1$. Now, there are $O(B)$ possible values of the left-hand side of \eqref{eq:powersum}, hence each value, including $C$, occurs on average $\approx K_2 B^{S-1}$ times. 
Thus, we have the following cases.
\begin{itemize}
\item If $S>1$, $K_2 B^{S-1}$ grows reasonably fast with $B$. Hence, if a solution cannot be easily found by a computer search, we expect that there should be a special reason (obstruction) to a solution’s existence, and that we can in principle discover this obstruction and use it to prove that there are no solutions.
\item If $S=1$, we expect a constant number of solutions in the range $B/2 < \max_i |a_i x_i^{d_i}| \leq B$. By summing up solutions in the intervals $[B/2,B]$,$[B/4,B/2]$, $[B/8,B/4]$, and so on, we conclude that there should be about $K_2 \log B$ solutions with $ \max_i |a_i x_i^{d_i}| \leq B$. If the constant $K_2$ is small, we expect that equation \eqref{eq:powersum} should have infinitely many solutions, but the smallest one may be astronomical. So, for this type of equations, it makes sense to do a deeper computer search aiming to find large solutions.
\item If $S<1$, then $K_2 B^{S-1}$ decreases fast with $B$. Summing up over intervals $[B,2B]$, $[2B,4B]$, and so on, we conclude that the total number of solutions with $\max_i |a_i x_i^{d_i}| > B$ is expected to be $K_2 B^{S-1} \sum_{i=1}^\infty 2^{i(S-1)} = K_3 B^{S-1}$ for some $K_3>0$. Hence, if the equation happen to have no small solutions, it may well have no solutions at all, even if there are no obstructions to the solution’s existence. The absence of obstructions may make the proof of the solution’s non-existence especially challenging.  
\end{itemize}

\begin{table}
\begin{center}
\begin{tabular}{ |c|c|c| } 
 \hline
 $H$ & Equation & Status \\ 
 \hline\hline
 $29$ & $1+x+x^3+x y^2-z+z^3 =0 $ &  Has an integer solution, see  \eqref{eq:h29dsol} \\
  & $3+x^2+y+x^2 y+y^2-y z^2 = 0$ &  No integer solutions, see Proposition  \ref{prob:h29c} \\ 
\hline
 $31$ & $3-x^2+x^2 y+y^2+y z-y z^2 = 0$ & No integer solutions, see Proposition \ref{prob:h31a} \\ 
  & $3+x^2+x^2 y+2 y^2-y z^2 = 0$ & No integer solutions, see Proposition \ref{prob:h31f} \\ 
  & $3+x^3 y+y^2+z^3 = 0$ & Open \\ 
\hline
 $32$ & $4+x^4+x y+y^3=0$ &  Open \\ 
  & $4+x+x^4+y+y^3=0$ & Open \\ 
  & $2+2 x-x^4-y+y^3=0$ & Open \\
  & $2+x+x^4+x y+y^3=0$ &  Open \\  
\hline
 $33$ & $1+x^4+x y-y^2+y^3 = 0$ & - \\ 
  & $3+x+x^4+x y+y^3 = 0$ & -  \\ 
  & $3+x-x^4+x y+y^3 = 0$ & -  \\
  & $3-x-x^4+x y+y^3 = 0$ & -  \\
  & $1-x+x^4+y+x y+y^3 = 0$ & -  \\
  & $1+2 x^2+x^2 y+2 y^2-y z^2 = 0$ & No integer solutions \\
  & $1+2 y+2 y^2+x^2 y z+z^2 = 0$ & No integer solutions \\
  & $3+x^3+x^2 y^2+z+z^2=0$ & $x=-21495$, $y=-146$, $z=287583$ \\ 
  & $3-x+x^3+x^2 y^2+z^2=0$ & $x=-13147$, $y=114$, $z=161553$ \\ 
  & $3+x^4-z+y^2 z+z^2 = 0$ & $x=708313$, $y=1099536$, $z=-267298981516$ \\
  & $1+x y+x^2 y+4 z+y^2 z-z^2 = 0$ & - \\
  & $1+x^2+2 y+x^2 y+y^2 z+y z^2 = 0$ & No integer solutions  \\
  & $3-y+x^2 y+y^2+x y z-2 z^2 = 0$ & -  \\
  & $1-x+x^3+x^2 y^2+z+z^2 = 0$ & -  \\
  & $1+2 x+x^3+x y+x y^2-2 z^2 = 0$ & -  \\
  & $1+x^2+x^3+y^2+x y^2-x z^2 = 0$ & No integer solutions  \\
  & $1+x-x^2+y^2-x^2 y^2+z+z^2 = 0$ & -  \\ 
  & $1-2 x-x^2+x^3-4 z+y^2 z  = 0$ & No integer solutions \\ 
  & $1-x+x^2+x^3+3 z+y z+y^2 z = 0$ & No integer solutions \\ 
  & $1+x-x^2+x^3+3 z+y z+y^2 z = 0$ & No integer solutions \\  
 \hline
\end{tabular}
\caption{\label{tab:H33unsolved} All equations of size $H\leq 33$ not solved by Algorithm \ref{alg:unified} with parameters $N_1=10^5$, $M=150$, $N_0=N_2=A=B=100$ and $S=2$.}
\end{center} 
\end{table}

For equation \eqref{eq:h31main}, we have $S=\frac{1}{2}+\frac{1}{3}+\frac{1}{6}=1$, hence there are two reasonable ways how one may try to solve this equation: either find a reason (obstruction) why an integer solution cannot exist, or use computer search to try to find a large solution.

One may also ask what are the next smallest open equations after \eqref{eq:h31main}? Table \ref{tab:H33unsolved} lists all equations up to $H\leq 33$ that Algorithm \ref{alg:unified} cannot solve automatically. As you can see, for $H=32$ all equations in $n\geq 3$ variables are solved automatically, but there are four open $2$-variable equations that we will discuss in the next section. We prefer to stop our analysis of the general equations at $H\leq 32$, and leave the equations of size $H=33$ for the interested reader to investigate. For the equations we know how to solve, the answer is given but the proof is not. Equations marked as ``$-$'' has not been even seriously attempted to solve. 

\section{Equations of special types}\label{sec:special}

%
%
%
%
%


\subsection{Classification by the number of variables}\label{eq:2var}

We may classify equations by the number of variables. Equations in one  variable are easy to solve, so the first interesting case are equations in two variables. 
The smallest two-variable equations Algorithm \ref{alg:unified} cannot solve are the equations
\begin{equation}\label{eq:H32a}
y^3+xy+x^4+4=0,
\end{equation}
\begin{equation}\label{eq:H32b}
y^3+xy+x^4+x+2=0,
\end{equation}
\begin{equation}\label{eq:H32c}
y^3+y=x^4+x+4
\end{equation}
and
\begin{equation}\label{eq:H32d}
y^3-y=x^4-2x-2
\end{equation}
of size $H=32$.

Equation \eqref{eq:H32a} looks similar to \eqref{eq:H31a}, but it is not clear how to solve it using Corollary \ref{cor:quadform}. If both $x$ and $y$ are even, substitution $x=2u$, $y=2v$ reduces the equation to $-v(2v^2+u)=(2u^2)^2+1$. The right hand side is a sum of squares and by Corollary \ref{cor:quadform} has only the prime factors in the form $p=4k+1$. However, an example $v=-1$, $u=3$ demonstrates that $v(2v^2+u)$ may also have only such prime factors. 

We leave the solvability of these equations to the reader as an open question. 

\begin{question}\label{qu:2var}
Determine whether each of the equations \eqref{eq:H32a}-\eqref{eq:H32d} have any integer solution.
\end{question}


Let us do some heuristic analysis for the equations \eqref{eq:H32a}-\eqref{eq:H32d}. They have monomials $x^4$, $y^3$, and some monomials of lower degree, which we will ignore. Then $S=\frac{1}{4}+\frac{1}{3}<1$, where $S$ in defined in \eqref{eq:Sdef}. 
%
This means that if such equations have no small solutions, they are likely to have no solutions at all, even if there are no clear obstructions for the solution’s existence. This is what makes such equations difficult.
%

On the other hand, these equations are in $2$ variables, and correspond to the study of integer points on curves. Compared to the study of integer points on surfaces, this question is much better understood. 
A natural parameter of measuring the complexity of a $2$-variable Diophantine equation
\begin{equation}\label{eq:2vargen}
P(x,y)=0
\end{equation} 
is genus. By Theorem \ref{th:notabsirr}, we may assume that $P(x,y)$ is absolutely irreducible. Then the set of all complex solutions to \eqref{eq:2vargen} form a (connected) surface. Recall that the \emph{genus} $g$ of a connected surface is the maximum number of cuttings that can be made along non-intersecting closed simple curves on the surface without making it disconnected. The genus–degree formula states that
\begin{equation}\label{eq:gendeg}
g \leq \frac{1}{2}(d-1)(d-2),
\end{equation} 
where $d$ is the degree of $P$. If $g=0$, then there are known (and practical) algorithms \cite{poulakis1993points, poulakis2002solving} for determining whether \eqref{eq:2vargen} has a finite or infinite number of integer solutions, and, in the former case, list them all. A classical 1929 Theorem of Siegel \cite{siegel1929uber} states that if $P(x,y)$ has genus $g\geq 1$, then \eqref{eq:2vargen} has a finite number of integer solutions. Hence, we have the following general result.

\begin{theorem}\label{th:finiteness2var}
There is an algorithm that, given a polynomial $P(x,y)$ with integer coefficients, decides whether the equation \eqref{eq:2vargen} 
has a finite or infinite number of integer solutions.
\end{theorem}

See the introduction of \cite{bilu2000diophantine} for an explanation how to deduce Theorem \ref{th:finiteness2var} from Siegel's Theorem \cite{siegel1929uber}. However, all the known proofs of Siegel's Theorem are ineffective. For a given $P$, the algorithm can output that \eqref{eq:2vargen} has a finite number of integer solutions, and even give an explicit upper bound for the number of such solutions \cite{bombieri1990mordell}, but not for their size. Hence, it cannot be used to find all the solutions and even to decide if any integer solution exists.

In 1970, Baker \cite{baker1970integer} proved an effective version of Siegel's Theorem for genus $1$ curves. Together with the genus $0$ algorithm \cite{poulakis1993points, poulakis2002solving}, this implies the following Theorem.

\begin{theorem}\label{th:gesus1}
There is an algorithm that, given an absolutely irreducible polynomial $P(x,y)$ of genus $g \leq 1$ with integer coefficients, determines all integer solutions to \eqref{eq:2vargen}.
\end{theorem}

No analogue of Theorem \ref{th:gesus1} is known for higher genus, even for $g=2$. As noted by Masser \cite{masser2020alan}, there is not even known algorithm that, given integer $a$ as an input, outputs all integer solution to the simple-looking genus $2$ equation
$$
x^4 - y^3 - a x y = 0.
$$
However, there are methods that works well for many specific genus $2$ equations with small coefficients. It is known that any genus $2$ equation can, after rational change of variables, be written in the form
\begin{equation}\label{eq:gen2}
y^2 = P(x),
\end{equation}
where $P(x)$ is a polynomial with integer coefficients of degree $5$ or $6$ with distinct roots. Note that integer solutions to \eqref{eq:gen2} can be computed by Theorem \ref{th:ellBaker1}, but the integrality of the solutions is not preserved after a rational change of variables, so we need to determine the rational solutions to \eqref{eq:gen2}. A celebrated theorem of Faltings \cite{faltings1983endlichkeitssatze} states that every equation of genus $g\geq 2$ has at most a finite number of rational solutions. All the known proofs of this theorem are ineffective and do not lead to an algorithm, but there are effective methods that work well for many individual equations. A notable example is the ``Chabauty'' function implemented in Magma \cite{bosma1997magma}, which is based on the method introduced by Chabauty \cite{chabauty1941points} and then refined and extended by many other authors. While this method alone cannot solve all genus $2$ equations, Poonen \cite{poonen2000computing} argued that its combination with other known methods seems to work for every individual genus $2$ equation, although it seems to be very difficult to prove any general theorem along this line. 
Algorithm \ref{alg:unified} has been able to automatically solve/exclude all genus $2$ equations up to $H\leq 50$.

By the genus–degree formula \eqref{eq:gendeg}, all quartic equations have genus at most $3$, and all open equations \eqref{eq:H32a}-\eqref{eq:H32d} have genus exactly $3$. 
There are some methods for solving the genus $3$ equations, see e.g. \cite{bruin2016generalized}, but the presently available algorithms work only for some special equations and seems to be not feasible for equations \eqref{eq:H32a}-\eqref{eq:H32d}.

With the exception of equation \eqref{eq:h34vieta}, all the equations we have discussed so far were in $2$ or $3$ variables. By Proposition \ref{prop:4nplus4}, in the range $H\leq 32$ it suffices to consider only equations in $n\leq 6$ variables. 
The program\footnote{The program is written in Wolfram Mathematica \cite{wolfram1991mathematica} and required 20 hours on a standard laptop to go through the full range of equations in $n\leq 6$ variables with $0\leq H \leq 32$.} 
returns no equations in $4\leq n\leq 6$ variables of size $H \leq 32$ that are not solvable by the trivial methods of Section \ref{sec:trivial}. The smallest equation in $n>3$ variables that requires application of Algorithm  \ref{alg:quadform} is the equation   
\begin{equation}\label{eq:h33var4}
y(x^2+z^2+2) = t^2 + (t+1)^2
\end{equation}
of size $H=33$, while the smallest equation that requires Algorithm  \ref{alg:vieta} is the equation \eqref{eq:h34vieta} of size $H=34$. Our general Algorithm \ref{alg:unified} has been able to automatically solve all $4$-variable equations of size $H\leq 36$. 

\subsection{Symmetric equations}

Symmetric equations are those that do not change with permutations of variables. They usually look particularly nice. We start with the case of $2$-variables equations and ignore Theorem \ref{th:Runge} for a moment. Then the smallest symmetric equation the program returned was the equation 
%
\begin{equation}\label{eq:H39sym2}
x^3 + x^2y^2 + y^3 = 7
\end{equation}
of size $H=39$. 
We remark that this equation is covered by Theorem \ref{th:Runge}, because it has $m=n=3$ and condition (C1) holds with $i=j=2$. 
However, it is instructive to present a direct proof that \eqref{eq:H39sym2} has no integer solutions due to Rouse \cite{J2020}.

\begin{proposition}\label{prop:H39sym2}
Equation \eqref{eq:H39sym2} has no integer solutions.
\end{proposition}
\begin{proof}
A standard tool for solving symmetric equations in $2$ variables is change of variables $u=-(x+y)$, $v=xy$. Then
$$
-u^3 = (x+y)^3 = x^3+y^3 + 3xy(x+y) = x^3+y^3-3uv, 
$$
hence $x^3+y^3 = -u^3+3uv$, and \eqref{eq:H39sym2} reduces to
$$
v^2 + 3uv = u^3+7. 
$$
This equation is an example of an elliptic curve in Weierstrass form \eqref{eq:ellWei}, and it can be solved by command \eqref{eq:commandellWei}.
In this example, the solutions are $(-3,4)$, $(-3,5)$, $(186,-2831)$ and $(186,2273)$.
If, for example, $u=-3$ and $v=4$, then $x+y=3$ and $xy=4$. Then $x(3-x)=4$, or $x^2-3x+4=0$. But this quadratic equation has no integer solutions. The other three cases can be checked similarly.
\end{proof}

%

After we excluded all equations covered by Theorem \ref{th:Runge}, the program returned no symmetric equations in $2$ variables of size $H \leq 50$. 

For $3$ or more variables, we have already discussed some symmetric equations such as 
\eqref{eq:h34vieta}, solvable by the Vieta jumping technique. After we exclude all equations covered by Algorithm \ref{alg:vieta}, the smallest symmetric equation the program returns is the equation
%
%
%
\begin{equation}\label{eq:H37sym3}
x^3 + y^3 + z^3 + xyz = 5
\end{equation}
of size $H=37$. 
This equation was left open in the first version of this paper \cite{grechuk2021diophantinev1}, but then we discovered that it has a solution
\begin{equation}\label{eq:H37symsol}
x =-3028982, \quad y=-3786648, \quad z=3480565.
\end{equation}
We remark that a direct approach of selecting a bound $B$, trying all pairs $(x,y)$ with $\max{|x|,|y|}\leq B$ and solving for $z$ would require checking over $10^{13}$ pairs before finding the solution \eqref{eq:H37symsol}. Instead, we applied the linear substitutions $y \to v-z$ and $x\to u+3v$ for new integer variables $u,v$ to reduce the equation to
$$
-5 + 28 v^3 + u^3 + 9 u^2 v + 27 u v^2 + u v z - u z^2 = 0.
$$
To find a solution to the last equation, we can try $v=0, \pm 1, \pm 2, \dots$, then choose $u$ among the divisors of $28v^3-5$, and solve for $z$. This method allowed us to find the the solution \eqref{eq:H37symsol} is a reasonable time on standard PC.

The next symmetric equation the program returns is the equation
\begin{equation}\label{eq:H39sym}
x^3+x+y^3+y+z^3+z = xyz + 1
\end{equation}
of size $H=39$. Using the same method as for equation \eqref{eq:H37sym3}, we were able to search for solutions to \eqref{eq:H39sym} up to $|y+z|\leq 10^8$, but found no solutions in this range.

\begin{question}
Do there exist integers $x,y,z$ satisfying \eqref{eq:H39sym}? 
\end{question}

In fact, \eqref{eq:H39sym} is the only remaining open symmetric equation of size $H\leq 42$. For this equation, we have $S=\frac{1}{3}+\frac{1}{3}+\frac{1}{3}=1$, where $S$ in defined in \eqref{eq:Sdef}. Hence, it is quite likely that 
it has integer solutions but the smallest one may happen to be quite large. 


\subsection{Equations with $3$ monomials}

Equation \eqref{eq:H39sym}, while symmetric and having reasonably low $H$, has $8$ monomials and therefore does not look particularly simple. Another possible measure of simplicity of an equation is the number of monomials. Note than one monomial must be a non-zero free term, otherwise $0$ is a solution. An equation with $2$ monomials has the form
\begin{equation}\label{eq:2mon}
a x_1^{k_1} \dots x_n ^{k_n} = b,
\end{equation}
where $a,b$ are non-zero integers and $k_1, \dots, k_n$ are positive integers. If either $\frac{b}{a}$ is not an integer, or $\frac{b}{a}<0$ and all $k_i$ are even, then  \eqref{eq:2mon} has no integer solution. Otherwise let $p_1^{m_1}\dots p_l^{m_l}$ be the prime factorization of $\left|\frac{b}{a}\right|$, and let $S(k_1, \dots, k_n)$ be the set of non-negative integers $N$ representable as $N=\sum_{i=1}^n c_i k_i$ for some non-negative integers $c_1, \dots, c_n$. Then \eqref{eq:2mon} is solvable in integers if and only if 
\begin{equation}\label{eq:2moncond}
m_j \in S(k_1, \dots, k_n), \quad j=1,\dots,l.
\end{equation}
Indeed, if $x_1, \dots, x_n$ is any solution to \eqref{eq:2mon}, then each $x_i$ can be written as $x_i=\prod_{j=1}^l p_j^{c_{ij}}$ for some integers $c_{ij}\geq 0$, hence $m_j=\sum_{i=1}^n c_{ij} k_i, j=1,\dots,l$, and \eqref{eq:2moncond} follows. The proof of the converse direction is similar.

We next consider equations with $3$ monomials.
We have already met several such equations, see \eqref{eq:h15} and \eqref{eq:H26xy}, but such equations, while they may be non-trivial to solve, are covered by the general algorithms in Corollary \ref{cor:brindza}. In fact, Corollary \ref{cor:brindza} implies the existence of an algorithm for solving an arbitrary $3$-monomial equation in $2$ variables.

\begin{proposition}\label{prop:3mon2var}
There exists an algorithm which, given any polynomial $P(x,y)$ with $3$ monomials and integer coefficients, determines whether equation $P(x,y)=0$ has an integer solution.
\end{proposition} 
\begin{proof}
We may assume that $P(0,0)\neq 0$, otherwise $x=y=0$ is a solution. Then $P(x,y)$ has the form
$$
P(x,y)=a_1 x^{k_1}y^{k_2} + a_2x^{m_1}y^{m_2}+a_3,
$$ 
where $a_1,a_2,a_3$ are non-zero integers and $k_1,k_2,m_1,m_2$ are non-negative integers. If both $k_1$ and $m_1$ are positive, then $P(x,y)=0$ only if $x$ is a divisor of $a_3$, and we may solve the equation by trying all the divisors and solving the resulting equations in one variable. The case when both $k_2$ and $m_2$ are positive is similar. Otherwise equation $P(x,y)=0$ can be written as
$$
a x^k = b y^{m} +c, 
$$ 
where we may assume that $a>0$ and $c\neq 0$.
But this equation is solvable in finite time by Corollary \ref{cor:brindza}.
\end{proof}

Given Proposition \ref{prop:3mon2var}, we may restrict our attention to $3$-monomial equations with at least $3$ variables. Among the equations of this form we would like to mention the equation
\begin{equation}\label{eq:H42term3b}
x^3y^2 = z^2 - 6
\end{equation}
of size $H=42$, for which Algorithm \ref{alg:largesol} returned a solution
$$
x=19, \quad y=755,031,379, \quad z=62,531,004,125.
$$ 
This solution would be impossible to find by a direct search, but the Algorithm noted that for each fixed $x$, equation \eqref{eq:H42term3b} is a quadratic equation in $2$ variables, which is easily solvable \cite{Alpern}. The algorithm then tried a few values of $x$, and, for $x=19$, found that the equation $19^3y^2 = z^2 - 6$ is solvable.

The smallest $3$-monomial equation left open by Algorithm \ref{alg:unified} is the equation
\begin{equation}\label{eq:H46term3}
x^3y^2 = z^3 + 6
\end{equation}
of size $H=46$.
The famous and well-believed abc conjecture of Masser and Oesterl\'e \cite{oesterle1988nouvelles} states that if $a,b,c$ are coprime positive integers such that $a+b=c$, then the product of the district prime factors of $abc$ is not much smaller than $c$. More precisely, if $\mathrm{rad}(n)$ denotes the product of district prime factors of $n$, then the abc conjecture predicts that for every real number $\epsilon>0$, there exists only a finite number of triples $(a,b,c)$ of coprime positive integers such that $a+b=c$ and
$$
c > \mathrm{rad}(abc)^{1+\epsilon}.
$$
It is easy to check that this conjecture implies that equation \eqref{eq:H46term3} may have at most finitely many integer solutions. However, the abc conjecture remains open, and even its truth would not exclude a possibility of \emph{some} solutions to \eqref{eq:H46term3}. We leave the solvability of this equation as an open question.
 
\begin{question}
Do there exist integers $x,y,z$ satisfying \eqref{eq:H46term3}? 
\end{question}

The following heuristic analysis shows that this equation may have no solutions but proving this may be quite difficult. For an integer $B>0$, there are about 
$O(B^{1/2+1/3})=O(B^{5/6})$ triples $(x,y,z)$ such that $B/2<\max\{|x^3y^2|,|z|^3\}\leq B$. Hence, if we search for a solution to \eqref{eq:H46term3} in this range, there will be about $O(B^{-1/6})$ expected solutions. This makes the existence of large solutions quite unlikely, even if there are no clear obstructions for the solution existence. In case of absence of such obstructions, proving that a surface has no integer points may be very difficult.

\subsection{The shortest open equations}\label{sec:short}

In this section we discuss what happens if, instead of ordering the equations by $H$, we order them by length $l$ defined in \eqref{eq:ldef}, or, equivalently, by an integer
$$
L(P) := 2^{l(P)} = \prod_{i=1}^k|a_i|\cdot 2^{\sum_{i=1}^k d_i}.
$$
Algorithm \ref{alg:unified} has been able to automatically solve all equations of length $l<10$, but left open a few equations of length $l=10$. We first mention the equation
$$
2+y^2+x^3 y^2+z^2 = 0
$$ 
that has not been solved by Algorithm \ref{alg:largesol} only because we have chosen parameter $N_2=100$. Going slightly beyond this limit returns a solution 
$$
x=-103, \quad y=56781721110114762679275339, \quad z=59355940474298270525410570738.
$$

Another equation of length $l=10$ that Algorithm \ref{alg:unified} could not solve was
$
1+x^3 y+x^2 y^2+z^2 = 0,
$
but non-linear substitution $-yz \to y$ (which the Algorithm did not try) transforms this equation into \eqref{eq:h17}.

The only equations of length $l\leq 10$ that seem to require a new idea are the equations
\begin{equation}\label{eq:l10main1}
y(x^3-y) = z^4+1,
\end{equation}
\begin{equation}\label{eq:l10main2}
2 y^3 + x y + x^4 + 1 = 0
\end{equation}
and 
\begin{equation}\label{eq:l10main3}
x^3 y^2 = z^4+2
\end{equation}
that have length $l=10$ and sizes $H$ equal to $37$, $37$, and $50$, respectively. Equation \eqref{eq:l10main1} resembles the current smallest open equation \eqref{eq:h31main}, but with $z^4$ instead of $z^3$. This difference seems to be crucial, because $S$ defined in \eqref{eq:Sdef} is now $S=\frac{1}{2}+\frac{1}{4}+\frac{1}{6} = \frac{11}{12} < 1$, hence the equation is unlikely to have large solutions. The same analysis is true for the equations \eqref{eq:l10main2} and \eqref{eq:l10main3} as well. On the other hand, we have checked that all these equations have no  solutions in the range up to a million. We leave their solubility as open questions.

\begin{question}\label{qu:shortest}
Do there exist integers $x,y,z$ satisfying \eqref{eq:l10main1}? 
\end{question}

\begin{question}\label{qu:shortest2}
Do there exist integers $x,y$ satisfying \eqref{eq:l10main2}? 
\end{question}

\begin{question}\label{qu:shortest3}
Do there exist integers $x,y,z$ satisfying \eqref{eq:l10main3}? 
\end{question}

As mentioned above, equation \eqref{eq:l10main1} is similar to \eqref{eq:h31main}. Further, equation \eqref{eq:l10main2} is similar to the smallest open $2$-variable equations \eqref{eq:H32a}-\eqref{eq:H32d}, while equation \eqref{eq:l10main3} resembles the smallest open $3$-monomial equation \eqref{eq:H46term3}. It is likely that if we resolve equations with small $l$, we may use similar methods to resolve the equations with the smallest $H$, and vice versa. In this sense, it does not really matter what ``measure of size'' we use to order equations.

\section{Conclusions}\label{sec:concl}


We have considered, in a systematic way, all ``small'' polynomial Diophantine equations and identified those for which the solvability problem (Problem \ref{prob:main}) is interesting. 
As expected, a vast majority of small equations either have small easy-to-find solutions or have no solutions for trivial reasons. However, we also have found some simple-to-state equations which are not solvable in integers for deeper reasons, as well as equations
whose smallest solutions are quite large and are difficult to find by a direct search. Some of the equations are solved, some others are left to the reader as exercises and open questions. 

An obvious direction for future research is to solve open questions listed here, proceed to the equations with higher values of $H$ or $l$, and see how far can we go. Table \ref{tab:open} summarizes all smallest open equations listed in this paper, just to have them all in one place. This project is very active, and some of the equations listed as open may be solved by the time you read this paper. Please check the mathoverflow question \cite{BG2021} for the up-to-date list of the current smallest open equations.

\begin{table}
\begin{center}
\begin{tabular}{ |c|c|c| } 
 \hline
 Equation & Size & Comment \\ 
 \hline\hline
 $y(x^3-y)=z^3+3$ & $H=31$ & The smallest (in $H$) open equation \\ 
\hline
 $y^3+xy+x^4+4=0$ & $H=32$ & The smallest open $2$-variable equations \\ 
 $y^3+xy+x^4+x+2=0$ & $H=32$ &  \\ 
 $y^3+y=x^4+x+4$ & $H=32$ & \\ 
 $y^3-y=x^4-2x-2$ & $H=32$ & \\ 
\hline
 $x^3+x+y^3+y+z^3+z = xyz + 1$ & $H=39$ & The smallest open symmetric equation \\ 
\hline
 $x^3y^2 = z^3 + 6$ & $H=46$ & The smallest open $3$-monomial equation \\ 
\hline
 $y(x^3-y) = z^4+1$ & $l=10$ & The shortest (in $l$) open equations \\ 
 $2 y^3 + x y + x^4 + 1 = 0$ & $l=10$ &  \\ 
 $x^3 y^2 = z^4+2$ & $l=10$ & \\ 
 \hline
\end{tabular}
\caption{\label{tab:open} The smallest open equations.}
\end{center} 
\end{table}

Another possible research direction is to investigate, in addition to Problem \ref{prob:main}, the following problems, in the increasing level of difficulty. Let ${\cal P}$ be the set of multivariate polynomials $P(x_1, \dots, x_n)$ with integer coefficients.

\begin{problem}\label{prob:nonzero}
Given $P \in {\cal P}$ with constant term $0$, does equation $P=0$ has a \emph{non-zero} integer solution?
\end{problem}

\begin{problem}\label{prob:finlist}
Given $P \in {\cal P}$, determine whether equation $P=0$ has a finite number of integer solutions, and if so, list them all. 
\end{problem}

\begin{problem}\label{prob:allsol}
Given $P \in {\cal P}$, find all integer solutions to equation $P=0$. If there are infinitely many of them, find a convenient way to describe them all (one can use expressions with parameters, recurrence relations, etc.)
\end{problem}

\begin{problem}\label{prob:rat}
Investigate the same questions as above for rational solutions.
\end{problem}

See \cite{grechuk2022on} for some initial investigations of Problems \ref{prob:finlist} and \ref{prob:allsol}. Because these problems look more difficult than Problem \ref{prob:main}, one is expected to find even smaller equations for which the listed problems are highly non-trivial. Such equations may serve as nice test cases for any known or new technique for determining integer and/or rational points on curves and surfaces. 



\section*{Acknowledgements}

I thank anonymous mathoverflow user Zidane whose question \cite{Z2018} inspired this work. I also thank other mathoverflow users for very helpful discussions on this topic, especially Will Sawin for solving equations \eqref{eq:h22mark}
and \eqref{eq:h26general}, Fedor Petrov for the alternative solution of \eqref{eq:h22mark}, Andrew R. Booker for finding a solution to \eqref{eq:h29d}, Majumdar and Sury for solving \eqref{eq:h29main} in a recent preprint \cite{majumdar2021fruit}, Victor Ostrik for outlining the solution of equation \eqref{eq:h17}, and Jeremy Rouse for solving equation \eqref{eq:H39sym2} (see Proposition \ref{prop:H39sym2}) and for pointing me to references \cite{poulakis1993points} and \cite{poulakis2002solving}. I thank University of Leicester students Daniel Bishop-Jennings, Sander Bharaj and Will Moreland for doing the projects on this topic under my supervision and independently verifying some of the results presented here. I thank William Gasarch for carefully reading an earlier draft of this paper and providing me with the detailed and helpful list of comments and suggestions. Last but not least, I thank Aubrey de Grey for many useful suggestions (including the suggestion to include the analysis with an alternative measure of the equations sizes), and for writing his own version of the computer program for enumerating equations, which provides an independent validation that none of the equations of small size has been accidentally missing. 

\bibliography{references} 
\bibliographystyle{plain}

\end{document}